\documentclass[reqno,12pt]{amsart}
 \usepackage{amssymb}
 \usepackage{latexsym}
 \usepackage{graphicx}
 \usepackage{multicol}
 \usepackage{mathrsfs}
 \usepackage{pstricks, pst-node}
 \usepackage{young}
 \usepackage[vcentermath,enableskew,stdtext]{youngtab}
 \usepackage[left=1.1in,right=1.1in,top=1.1in,bottom=1.14in]{geometry}
 \usepackage[all]{xy}
    \SelectTips{cm}{10}     %use the nicer arrowheads
    \everyxy={<2.5em,0em>:} % Sets the scale I like
 \usepackage{fancyhdr}      %%%%%%%%%%%% Pagestyle stuff %%%%%%%%%%%%%%%
 \usepackage{multicol}
 \usepackage{multirow}
 \usepackage{enumitem}
 \usepackage{stmaryrd}
 \usepackage{etex}
 \usepackage{tikz}
 \usepackage{float}
 \restylefloat{figure}
 \usetikzlibrary{shapes}
 \usetikzlibrary{positioning}

\newcounter{ctr}
\theoremstyle{plain}
\newtheorem{theorem}{Theorem}[section]

\newtheorem*{lemma*}{Lemma}
\newtheorem{lemma}[theorem]{Lemma}
\newtheorem{corollary}[theorem]{Corollary}
\newtheorem{proposition}[theorem]{Proposition}
\newtheorem{conjecture}[theorem]{Conjecture}

\newtheorem{propdef}[theorem]{Proposition-Definition}

\theoremstyle{definition}
\newtheorem{definition}[theorem]{Definition}

\newtheorem{remark}[theorem]{Remark}
\newtheorem{example}[theorem]{Example}

\newcommand{\B}{\ensuremath{\mathscr{B}}}

\newcommand{\CC}{\ensuremath{\mathbb{C}}}

\newcommand{\F}{\ensuremath{\mathscr{F}}}

\newcommand{\gl}{\ensuremath{\mathfrak{gl}}}

\renewcommand{\H}{\ensuremath{\mathscr{H}}}

\renewcommand{\L}{\ensuremath{\mathscr{L}}}

\renewcommand{\O}{\ensuremath{\mathscr{O}}}

\newcommand{\QQ}{\ensuremath{\mathbb{Q}}}
\newcommand{\R}{\ensuremath{\mathcal{R}}}

\renewcommand{\sl}{\ensuremath{\mathfrak{sl}}}

\newcommand{\ZZ}{\ensuremath{\mathbb{Z}}}
\newcommand{\liftT}{\ensuremath{\tilde{T}}}

\newcommand{\be}{\begin{equation}}
\newcommand{\ee}{\end{equation}}
\newcommand{\Res}{\text{\rm Res}}
\renewcommand{\S}{\ensuremath{\mathcal{S}}}
\newcommand{\tsr}{\ensuremath{\otimes}}
\newcommand{\C}{\ensuremath{C^{\prime}}} %canonical basis
\newcommand{\cp}{\ensuremath{c^{\prime}}} %lower canonical basis

\newcommand{\liftCp}{\ensuremath{\tilde{C}^{\prime}}}
\newcommand{\liftC}{\ensuremath{\tilde{C}}}
\newcommand{\liftc}{\ensuremath{\tilde{c}}}
\newcommand{\liftcp}{\ensuremath{\tilde{c}^{\prime}}} %lifted lower canonical basis

\newcommand{\br}[1]{\ensuremath{\overline{#1}}}
\newcommand{\nsbr}[1]{{\ensuremath{\check{#1}}}}

\renewcommand{\u}{\ensuremath{u}}  %q^{1/2}
\newcommand{\ui}{\ensuremath{u^{-1}}} %q^{-1/2}

\newcommand{\leftexp}[2]{{\vphantom{#2}}^{#1}{#2}}

 \newcommand{\dual}[1]{\ensuremath{{#1}^\vee}}

\newcommand{\y}{\ensuremath{Y}} %the Bernstein basis

\newcommand{\gd}{\ensuremath{\trianglerighteq}}
\newcommand{\gdneq}{\ensuremath{\triangleright}}
\newcommand{\ld}{\ensuremath{\trianglelefteq}}
\newcommand{\ldneq}{\ensuremath{\triangleleft}}
\newcommand{\tto}{\ensuremath{\rightsquigarrow}}

\newcommand{\sh}{\text{\rm sh}}

\newcommand{\transpose}[1]{{#1}^t}

\newcommand{\nsH}{\nsbr{\H}}

\newcommand{\bv}{\mathbf{v}}

\newcommand{\sP}{\mathcal{P}}

\newcommand{\field}{\ensuremath{K}}
\newcommand{\klo}[1]{\ensuremath{\leq_{#1}}}
\newcommand{\klocov}[1]{\ensuremath{\preceq_{#1}}}
\newcommand{\kloneq}[1]{\ensuremath{<_{#1}}}

\newcommand{\bT}{\ensuremath{\mathbf{T}}}
\newcommand{\bU}{\ensuremath{\mathbf{U}}}
\newcommand{\gt}{\ensuremath{C^\text{sn}}} % Gelfand-Tsetlin basis
\newcommand{\gtp}{\ensuremath{C^{\prime\text{sn}}}} % Gelfand-Tsetlin basis
\DeclareMathOperator{\sort}{sort}

\newcommand{\heart}{\heartsuit}

\newcommand{\myvcenter}[1]{\ensuremath{\vcenter{\hbox{#1}}}}
\newcommand{\crystalu}[1]{\ensuremath{\tilde{#1}^{\hspace{-2pt}\text{up}}}}
\newcommand{\crystall}[1]{\ensuremath{\tilde{#1}^{\hspace{-2pt}\text{low}}}}
\newcommand{\Uq}{\bU}
\newcommand{\QQA}{\ensuremath{\QQ[\u,\ui]}}
\newcommand{\Oint}{\ensuremath{\O_{\text{int}}^{\geq 0}}}

\newlength{\cellsizeCol}
\newlength{\cellsize}
\newcommand\tableau[1]{
\vcenter{
\let\\=\cr
\baselineskip=-16000pt \lineskiplimit=16000pt \lineskip=0pt
\halign{&\tableaucell{##}\cr#1\crcr}}}

\newcommand{\tableaucell}[1]{{%
\def \arg{#1}\def \void{}%
\ifx \void \arg
\vbox to \cellsize{\vfil \hrule width \cellsize height 0pt}%
\else \unitlength=\cellsize
\begin{picture}(1,1)
\put(0,0){\makebox(1,1){$#1$}}
\put(0,0){\line(1,0){1}}
\put(0,1){\line(1,0){1}}
\put(0,0){\line(0,1){1}}
\put(1,0){\line(0,1){1}}
\end{picture}%
\fi}}

\newcommand\pad[1]{
\vtop{
\let\\=\cr
\baselineskip=-16000pt
\lineskiplimit=16000pt
\lineskip=0pt
\halign{& \inviscell{##} \cr #1 \crcr} }
\hspace{-.73ex}}

\newcommand{\inviscell}[1]{{%
\unitlength=\cellsizeCol
\begin{picture}(1,1)
\put(0,0){\makebox(1,1){$#1$}}
\end{picture}%
}}

\begin{document}

\address{Department of Mathematics, Drexel University, Philadelphia, PA 19104}
\email{jblasiak@gmail.com} 

\author{Jonah Blasiak}\thanks{The author is currently an NSF postdoctoral fellow.}
%\title{Quantum Schur-Weyl duality and transition matrices between the  $C_w$'s and  $\C_w$'s}
\title{Quantum Schur-Weyl duality and projected canonical bases}
\begin{abstract}
Let  $\H_r$ be the generic type $A$ Hecke algebra defined over $\ZZ[\u, \ui]$.
The Kazhdan-Lusztig bases $\{C_w\}_{w \in \S_r}$ and $\{\C_w\}_{w \in \S_r}$ of $\H_r$ give rise to two different bases of the Specht module $M_\lambda$,  $\lambda \vdash r$, of $\H_r$. These bases are not equivalent and we show that the transition matrix  $S(\lambda)$ between the two is the identity at $\u = 0$ and $\u = \infty$.  To prove this, we first prove a similar property for the transition matrices  $\liftT, \liftT'$ between the Kazhdan-Lusztig bases and their projected counterparts $\{\liftC_w\}_{w \in \S_r},$ $\{\liftCp_w\}_{w \in \S_r}$, where $\liftC_w := C_w p_\lambda$, $\liftCp_w := \C_w p_\lambda$ and  $p_\lambda$ is the minimal central idempotent corresponding to the two-sided cell containing  $w$.  We prove this property of $\liftT,\liftT'$ using quantum Schur-Weyl duality and results about the upper and lower canonical basis of  $V^{\tsr r}$ ($V$ the natural representation of $U_q(\gl_n)$) from \cite{GL, FKK, Brundan}.  We also conjecture that the entries of  $S(\lambda)$ have a certain positivity property.
\end{abstract}
\maketitle
%??? abstract recently changed Dec 2013 so needs updating in arxiv
%Abstract for arXiv:
%Let  \H_r be the generic type A Hecke algebra defined over \ZZ[u, u^{-1}].
%The Kazhdan-Lusztig bases \{C_w\}_{w \in \S_r} and \{C'_w\}_{w \in \S_r} of \H_r give rise to two different bases of the Specht module M_\lambda,  \lambda \vdash r, of \H_r. These bases are not equivalent and we show that the transition matrix  S(\lambda) between the two is the identity at u = 0 and u = \infty.  To prove this, we first prove a similar property for the transition matrices  \tilde{T}, \tilde{T}' between the Kazhdan-Lusztig bases and their projected counterparts \{\tilde{C}_w\}_{w \in \S_r}, \{\tilde{C}'_w\}_{w \in \S_r}, where \tilde{C}_w := C_w p_\lambda, \tilde{C}'_w := C'_w p_\lambda and  p_\lambda is the minimal central idempotent corresponding to the two-sided cell containing  w.  We prove this property of \tilde{T},\tilde{T}' using quantum Schur-Weyl duality and results about the upper and lower canonical basis of  V^{\tsr r} (V the natural representation of U_q(\gl_n)) from \cite{GL, FKK, Brundan}.  We also conjecture that the entries of  S(\lambda) have a certain positivity property.

\section{Introduction}
\label{s Introduction}
Let $\{C_w: w \in \S_r\}$ and $\{\C_w: w \in \S_r\}$ be the Kazhdan-Lusztig bases of the type $A$ Hecke algebra $\H_r$, which we refer to as the upper and lower canonical basis of $\H_r$, respectively. After working with these bases for a while, we have convinced ourselves that it is not particularly useful to look at both at once---one can work with one or the other and it is easy to go back and forth between the two (precisely, there is an automorphism $\theta$ of $\H_r$ such that $\theta(\C_w) = (-1)^{\ell(w)}C_w$).
However, our recent work on the nonstandard Hecke algebra  \cite{Bnsbraid, B4} has forced us to look at both these bases simultaneously.  Before explaining how this comes about, let us describe our results and conjectures.

Let $\field = \QQ(\u)$, where $\u$ is the Hecke algebra parameter, and let $M_\lambda$ be the $\field \H_r$-irreducible of shape $\lambda \vdash r$.
The upper and lower canonical basis of $\H_r$ give rise to bases $\{C_Q : Q \in \text{SYT}(\lambda)\}$ and $\{\C_Q : Q \in \text{SYT}(\lambda)\}$ of $M_\lambda$, which we refer to as the upper and lower canonical basis of $M_\lambda$. These bases are not equivalent, and it appears to be a difficult and interesting question to understand the transition matrix  $S(\lambda)$ between them (which is well-defined up to a global scale by the irreducibility of  $ M_\lambda$).  It turns out that  $S(\lambda)$ is the identity at $\u =0$ and $\u = \infty$ (Theorem \ref{t transition C' to C}) and, though it is not completely clear what it should mean for  an element of $\field$ to be nonnegative, its entries appear to have some kind of nonnegativity (see Conjecture \ref{cj non-negativity T T' D}).

To compare the upper and lower canonical basis of $ M_\lambda$, we compare them both to certain seminormal bases of $M_\lambda$ in the sense of \cite{RamSeminormal}.  These bases are compatible with restriction along the chain of subalgebras $\H_1 \subseteq \cdots \subseteq \H_{r-1} \subseteq \H_r$ (see Definition \ref{d seminormal}).  Specifically, we define an upper (resp. lower) seminormal basis which differs from the upper (resp.  lower) canonical basis by a unitriangular transition matrix  $T(\lambda)$ (resp.  $T'(\lambda)$).  It appears that these transition matrices also possess some kind of nonnegativity property.  Since the restrictions $\H_{i-1} \subseteq \H_i$ are multiplicity-free, these seminormal bases differ from each other by a diagonal transformation $D(\lambda)$. Hence we have $S(\lambda) = T(\lambda) D(\lambda) T'(\lambda)^{-1}$.

We briefly mention some related investigations in the literature.  Other seminormal bases of $M_\lambda$ have been defined---for instance, Hoefsmit, and later, independently, Ocneanu, and Wenzl construct a Hecke algebra analog of Young's orthogonal basis (see \cite{Wenzl}).  This basis differs from our upper and lower seminormal bases by a diagonal transformation, but is not equal to either.  The recent paper \cite{GLS} uses an interpretation of the lower seminormal basis in terms of non-symmetric Macdonald polynomials to study $T'(\lambda)$ for $\lambda$ a two-row shape and gives an explicit formula for a column of this matrix (see Remark \ref{r Lascoux's paper}).
Along similar lines, the transition matrix between the upper canonical basis at $\u =1$ and Young's natural basis of  $M_\lambda$  is studied by Garsia and McLarnan in \cite{GM}; they show that this matrix is unitriangular and has integer entries.

Our investigation further involves projecting the basis element $C_w$ (resp.  $\C_w$) onto the isotypic component corresponding to the two-sided cell containing $w$.  This results in what we call the projected upper (resp.  lower) canonical basis; let $\liftT$ (resp.  $\liftT'$) denote the transition matrix between the projected and upper (resp. lower) canonical basis. The properties we end up proving about  $S(\lambda), T(\lambda), T'(\lambda)$ all follow from properties of $\liftT$ and  $\liftT'$.  And we are able to get some handle on  $\liftT$ and  $\liftT'$ using quantum Schur-Weyl duality. Specifically, we use the compatibility between an upper (resp.  lower) canonical basis of $V^{\tsr r}$ with the upper (resp.  lower) canonical basis of  $\H_r$ and well-known results about crystal lattices, where $V$ is the natural representation of $U_q(\gl_n)$. The results we need are similar to those in \cite{GL, FKK, Brundan} and follow easily from results of \cite{LBook,Kas2}.  Brundan's paper \cite{Brundan} is particularly well adapted to our needs and we follow it closely.

We now return to our original motivation. The type $A$ \emph{nonstandard Hecke algebra}  $\nsH_r$ is the subalgebra of $\H_r \tsr \H_r$ generated by the elements
\[ \sP_s := \C_s \tsr \C_s + C_s \tsr C_s, \ s \in S, \]
where $S = \{s_1, \dots, s_{r-1}\}$ is the set of simple reflections of $\S_r$.
We think of the inclusion $\nsH_r \hookrightarrow   \H_r \tsr \H_r$ as a deformation of the coproduct $\ZZ \S_r \to \ZZ \S_r \tsr \ZZ \S_r$,  $w \mapsto w \tsr w$.  This algebra was constructed by Mulmuley and Sohoni in \cite{GCT4} in an attempt to use canonical bases to understand Kronecker coefficients.

Let $\epsilon_+ = M_{(r)}$, $\epsilon_- = M_{(1^r)}$ be the trivial and sign representations of $\field \H_r$. Any representation $M_\lambda \tsr M_\mu$ of $\field (\H_r \tsr \H_r)$ is a $\field \nsH_r$-module by restriction. The trivial and sign representations $\nsbr{\epsilon}_+$ and $\nsbr{\epsilon}_-$ of $\field \nsH_r$ are the restrictions of $\epsilon_+ \tsr \epsilon_+$ and $\epsilon_+ \tsr \epsilon_-$, respectively.
There is a single copy of $\nsbr{\epsilon}_+$ inside $\Res_{\field \nsH_r}M_\lambda \tsr M_\lambda$ and a single copy of $\nsbr{\epsilon}_-$ inside $\Res_{\field \nsH_r}M_\lambda \tsr M_{\lambda'}$,  where $\lambda'$ is the conjugate partition of $\lambda$.
These can be written in terms canonical bases as
\be
\nsbr{\epsilon}_+ \cong \field \sum_{Q \in \text{SYT}(\lambda)} C_Q \tsr \C_Q, \quad\quad \nsbr{\epsilon}_- \cong \field \sum_{Q \in \text{SYT}(\lambda)} (-1)^{\ell(Q)} C_Q \tsr C_{\transpose{Q}},
\ee
where $\transpose{Q}$ denotes the transpose of the SYT $Q$ and $\ell(Q)$ denotes the distance between $Q$ and some fixed tableau of shape  $\lambda$ in the dual Knuth equivalence graph on SYT$(\lambda)$.

An important part of understanding the nonstandard Hecke algebra is to understand its trivial and sign representations. If we fix a basis of  $\field (\H_r \tsr \H_r)$, say $\{C_v \tsr C_w:v,w \in \S_r\}$, then  expressing the central idempotent for $ \nsbr{\epsilon}_-$ in this basis involves understanding  $\liftT$ and expressing the central idempotent for $ \nsbr{\epsilon}_+$ involves understanding  $\liftT$ and  $S(\lambda)$.  The same difficulties come up if we choose the basis $\{C_v \tsr \C_w:v,w \in \S_r\}$. Admittedly, $\nsbr{\epsilon}_+ \subseteq \Res_{\field \nsH_r}M_\lambda \tsr M_\lambda$ and $\nsbr{\epsilon}_- \subseteq \Res_{\field \nsH_r}M_\lambda \tsr M_{\lambda'}$ both have simple expressions in terms of the Hecke orthogonal basis of \cite{Wenzl}.  However, we suspect it will be useful to understand $\nsH_r$ in terms of a basis like $\{C_v \tsr C_w:v,w \in \S_r\}$.  This is somewhat justified by our work in progress \cite{BMSGCT4}, joint with Ketan Mulmuley and Milind Sohoni, in which we use canonical bases of quantum groups to give a combinatorial rule for Kronecker coefficients with two two-row shapes (here, we do not need $S(\lambda)$, but the projected upper canonical basis plays an essential role).

%Roughly speaking, this is important because it is useful to take into account the symmetry  $g_{\lambda \mu \nu} = g_{\lambda' \mu' \nu}$ in Kronecker coefficients.  Another potential reason to understand this transition matrix would be to produce a canonical basis for the Hecke algebra version of the alternating group \cite{}.

This paper is organized as follows. In \textsection\ref{s Preliminaries and notation}--\ref{s Crystal bases of the quantized enveloping algebra} we introduce the necessary background on canonical bases of Hecke algebras and quantum groups.  We then use this in  \textsection\ref{s Quantum Schur-Weyl duality and canonical bases} to construct canonical bases of $V^{\tsr r}$ and relate them to those of $\H_r$, closely following \cite{Brundan}.  Next, in  \textsection\ref{s projected canonical bases}, we give several characterizations of projected canonical basis elements, which we then use in  \textsection\ref{s Consequences for the canonical bases of M_lambda} to prove that the transition matrices $S(\lambda),T(\lambda)$, and $T'(\lambda)$ are the identity at $\u=0$ and $\u = \infty$. Finally, in  \textsection\ref{s The two-row case}, we compute explicitly a matrix similar to $T'(\lambda)$, for $\lambda$ a two-row shape, using the  $U_q(\sl_2)$ graphical calculus of \cite{FK}.
\section{Preliminaries and notation}
Here we introduce notation for general Coxeter groups and then specialize to the weight lattice and Weyl group of $\gl_n$. In preparation for quantum Schur-Weyl duality, we introduce notation for words and tableaux.  Finally, we define cells in the general setting of modules with basis, rather than only for $W$-graphs.
\label{s Preliminaries and notation}
\subsection{General notation}
We work primarily over the ground rings $\mathbf{A} = \ZZ[\u, \ui]$ and $\field = \QQ(\u)$. Define $\field_0$ (resp. $\field_\infty$) to be the subring of $\field$ consisting of rational functions with no pole at $\u = 0$ (resp. $\u = \infty$).

Let $\br{\cdot}$ be the involution of $\field$ determined by $\br{u} = \ui$; it restricts to an involution of $\mathbf{A}$.
For a nonnegative integer $k$, the $\br{\cdot}$-invariant quantum integer is $[k] := \frac{\u^k - \u^{-k}}{\u - \ui} \in \mathbf{A}$ and the quantum factorial is $[k]! := [k][k-1]\dots[1]$.  We also use the notation $[k]$ to denote the set $\{1,\ldots,k\}$, but these usages should be easy to distinguish from context.

Let $(W, S)$ be a Coxeter group with length function $\ell$ and Bruhat order $<$. If $\ell(vw)=\ell(v)+\ell(w)$, then $vw = v\cdot w$ is a \emph{reduced factorization}. The \emph{right descent set} of $w \in W$ is $R(w) = \{s\in S : ws < w\}$.

For any $J\subseteq S$, the \emph{parabolic subgroup} $W_J$ is the subgroup of $W$ generated by $J$. Each left (resp. right) coset $wW_J$ (resp. $W_Jw$) contains a unique element of minimal length called a minimal coset representative. The set of all such elements is denoted $W^J$ (resp. $\leftexp{J}W$).

\subsection{Words and tableaux}
\label{ss type A combinatorics preliminaries}
Our results depend heavily on quantum Schur-Weyl duality, so we work almost entirely in type $A$.
The \emph{weight lattice} $X$ of the Lie algebra $\gl_n$ is $\ZZ^n$ with standard basis $\epsilon_1, \dots, \epsilon_{n}$. Its dual, $\dual{X}$, has basis $\dual{\epsilon}_{1}, \dots, \dual{\epsilon}_{n}$ dual to the standard. The simple roots are $\alpha_i = \epsilon_i - \epsilon_{i+1}, i \in [n-1]$.
We write $\lambda \vdash_l r$ for a partition $\lambda = (\lambda_1, \ldots, \lambda_l)$ of size $r = |\lambda| := \sum_{i=1}^l \lambda_i$. A partition $\lambda \vdash_n r$ is identified with the weight $\lambda_1\epsilon_1 + \dots + \lambda_n\epsilon_n \in X$.
%We also let $\alpha \vDash_l^{d_X} r$ denote a composition $\alpha = (\alpha_1,\ldots,\alpha_l)$ of $r$ with $\alpha_i \in [d_X]$.

For  $\zeta = (\zeta_1,\ldots,\zeta_l)$ a weak composition of $r$, let   $B_j$ be the interval $[\sum_{i=1}^{j-1}\zeta_i+1,\sum_{i=1}^{j}\zeta_i]$, $j \in [l]$.  Define  $J_\zeta = \{s_i:i, i +1\in B_j \text{ for some } j\}$ so that $(\S_r)_{J_\zeta} \cong \S_{\zeta_1} \times \dots \times \S_{\zeta_{l}}$.

%The set of dominant weights $X_+$ is the cone in $X$ given by
%\be X_+ = \{ \lambda \in X: \langle \lambda,\alpha'^\vee_i\rangle \geq 0 \text{ for all } i \}. \ee
%

Let $\mathbf{k} = k_1 k_2\dots k_r \in [n]^r$ be a word of length $r$ in the alphabet $[n]$. The \emph{content} of $\mathbf{k}$ is the tuple $(\zeta_1,\ldots,\zeta_n)$ whose  $i$-th entry $\zeta_i$ is the number of  $i$'s in $\mathbf{k}$.  The notation $\mathbf{k}^\dagger$ denotes the word $k_r k_{r-1} \dots k_1$.
The symmetric group $\S_r$ acts on $[n]^r$ on the right by $\mathbf{k} s_i = k_1 \dots k_{i-1} \,  k_{i+1} \, k_i \, k_{i+2} \dots k_r$.
Define $\sort(\mathbf{k})$ to be the tuple obtained by rearranging the $k_j$ in weakly increasing order.
For a word $\mathbf{k}$ of content $\zeta$, define $d(\mathbf{k})$ (resp.  $D(\mathbf{k})$) to be the element  $w$ of $\leftexp{J_\zeta}{\S_r}$ (resp.  $(w_0)_{J_\zeta} \leftexp{J_\zeta}{\S_r}$ where $(w_0)_{J_\zeta}$ is the longest element of $(\S_r)_{J_\zeta}$) such that $\sort(\mathbf{k})w = \mathbf{k}$.

The set of standard Young tableaux is denoted SYT, those SYT of size $r$ denoted SYT$^r$, those SYT$^r$ with at most $n$ rows denoted SYT$^r_{\leq n}$, and those SYT of shape $\lambda$ denoted SYT$(\lambda)$. The set of semistandard Young tableaux of size $r$ with entries in $[n]$ is denoted SSYT$_{[n]}^r$ and the subset of SSYT$_{[n]}^r$ of shape $\lambda \vdash r$ is SSYT$_{[n]}^r(\lambda)$.
Tableaux are drawn in English notation, so that entries of a SSYT strictly increase from north to south along columns and weakly increase from west to east along rows. For a tableau $T$, $|T|$ is the number of squares in $T$ and $\sh(T)$ its shape.

We let $P(\mathbf{k}), Q(\mathbf{k})$ denote the insertion and recording tableaux produced by the Robinson-Schensted-Knuth (RSK) algorithm applied to the word $\mathbf{k}$.  We abbreviate $\sh(P(\mathbf{k}))$ simply by $\sh(\mathbf{k})$.
Let $Z_\lambda$ be the superstandard tableau of shape and content $\lambda$---the tableau whose $i$-th row is filled with $i$'s.
The conjugate partition $\lambda'$ of a partition $\lambda$ is the partition whose diagram is the transpose of that of $\lambda$ and
$\transpose{Q}$ denotes the transpose of a SYT $Q$, so that $\sh(\transpose{Q}) = \sh(Q)'$. Lastly, $Q^\dagger$ denotes the Sch\"utzenberger involution of a SYT $Q$ (see, e.g., \cite[A1.2]{F}).
\subsection{Cells}
\label{ss cells}
We define cells in the general setting of modules with basis.
Let  $H$ be an $R$-algebra for some commutative ring $R$.
Let $M$ be a left $H$-module and $\Gamma$ an  $R$-basis of  $M$. The preorder $\klo{\Gamma}$ (also denoted $\klo{M}$) on the vertex set $\Gamma$ is generated by the relations
\be
\label{e preorder}
\delta\klocov{\Gamma}\gamma \begin{array}{c}\text{if there is an $h\in H$ such that $\delta$ appears with non-zero}\\ \text{coefficient in the expansion of $h\gamma$
in the basis $\Gamma$}. \end{array}
\ee

Equivalence classes of $\klo{\Gamma}$ are the \emph{left cells} of $(M, \Gamma)$. The preorder $\klo{M}$ induces a partial order on the left cells of $M$, which is also denoted $\klo{M}$.

A \emph{cellular submodule} of $(M, \Gamma)$ is a submodule of $M$ that is spanned by a subset of $\Gamma$ (and is necessarily a union of left cells). A \emph{cellular quotient} of $(M,\Gamma)$ is a quotient of $M$ by a cellular submodule, and a \emph{cellular subquotient} of $(M, \Gamma)$ is a cellular quotient of a cellular submodule.
%This is the same as the cellular submodule of a cellular quotient
We denote a cellular subquotient $R\Gamma'/ R\Gamma''$ by $R\Lambda$, where $\Gamma'' \subseteq \Gamma' \subseteq \Gamma$ span cellular submodules and $\Lambda = \Gamma' \setminus \Gamma''$.
We say that the left cells  $\Lambda$ and  $\Lambda'$ are isomorphic if $(R \Lambda, \Lambda)$ and $(R \Lambda', \Lambda')$ are isomorphic as modules with basis.

%\begin{remark} \label{r edge direction}
%Throughout this paper we use the convention that when identifying a poset with a directed acyclic graph, edges are directed from bigger elements to smaller ones.
%\end{remark}
Sometimes we speak of the left cells of $M$, cellular submodules of $M$, etc. or left cells of $\Gamma$, cellular submodules of $\Gamma$, etc. if the pair $(M, \Gamma)$ is clear from context.
For a right $H$-module  $M$, the \emph{right cells}, \emph{cellular submodules}, etc. of  $M$ are defined similarly with $\gamma h$ in place of $h \gamma$ in \eqref{e preorder}. We also use the terminology $H$-cells, $H$-cellular submodules, etc. to make it clear that the algebra $H$ is acting, and we omit left and right when they are clear.
\section{Hecke algebras and canonical bases}
\label{s Canonical bases of the type A Hecke algebra}
The \emph{Hecke algebra} $\H(W)$ of $(W, S)$ is the free $\mathbf{A}$-module with standard basis $\{T_w :\ w\in W\}$ and relations generated by
\be \label{e Hecke algebra def} \begin{array}{ll}T_vT_w = T_{vw} & \text{if } vw = v\cdot w\ \text{is a reduced factorization},\\
(T_s - \u)(T_s + \ui) = 0 & \text{if } s\in S.\end{array}\ee

For each $J\subseteq S$, $\H(W)_J$ denotes the subalgebra of $\H(W)$ with $\mathbf{A}$-basis $\{T_w:\ w\in W_J\}$, which is isomorphic to $\H(W_J)$.
%We abbreviate the restriction functor $\Res_{\H_J}:\H$-$\Mod\to \H_J$-$\Mod$ by  $\Res_J$.

In this section we recall the definition of the Kazhdan-Lusztig basis elements $C_w$ and $\C_w$ of \cite{KL} and some of their basic properties. We record some useful results about how they behave under induction and restriction.  Then we specialize to type  $A$ and review the beautiful connection between cells and the RSK algorithm.

\subsection{The upper and lower canonical basis of $\H(W)$}
The \emph{bar-involution}, $\br{\cdot}$, of $\H(W)$ is the additive map from $\H(W)$ to itself extending the $\br{\cdot}$-involution of $\mathbf{A}$ and satisfying $\br{T_w} = T_{w^{-1}}^{-1}$. Observe that $\br{T_{s}} = T_s^{-1} = T_s + \ui - u$ for $s \in S$. Some simple $\br{\cdot}$-invariant elements of $\H(W)$ are $\C_\text{id} := T_\text{id}$, $C_s := T_s - \u = T_s^{-1} - \ui$, and $\C_s := T_s + \ui = T_s^{-1} + u$, $s\in S$.

Define the lattices $\H(W)_{\ZZ[\u]} := \ZZ[\u] \{ T_w : w \in W \}$ and $\H(W)_{\ZZ[\ui]} := \ZZ[\ui] \{ T_w : w \in W \}$ of $\H(W)$.
\refstepcounter{equation}
\begin{enumerate}[label={(\theequation)}]
\item For each $w \in W$, there is  a unique element $C_w \in \H(W)$ such that $\br{C_w} = C_w$ and $C_w$ is congruent to $T_w \mod \u \H(W)_{\ZZ[\u]}$.
\end{enumerate}
The $\mathbf{A}$-basis $\Gamma_W := \{C_w: w\in W\}$
is the \emph{upper canonical basis} of $\H(W)$ (we use this language to be consistent with that for crystal bases).
Similarly,
\refstepcounter{equation}
\begin{enumerate}[label={(\theequation)}]
\item  for each $w \in W$, there is  a unique element $\C_w \in \H(W)$ such that $\br{\C_w} = \C_w$ and $\C_w$ is congruent to $T_w \mod \ui \H(W)_{\ZZ[\ui]}$.
\end{enumerate}
The $\mathbf{A}$-basis
$\Gamma'_W := \{\C_w : w \in W \}$ is the \emph{lower canonical basis} of $\H(W)$.

The coefficients of the lower canonical basis in terms of the standard basis are the \emph{Kazhdan-Lusztig polynomials} $P'_{x,w}$:
\be \C_w = \sum_{x \in W} P'_{x,w} T_x. \ee
(Our $P'_{x,w}$ are equal to $q^{(\ell(x)-\ell(w))/2}P_{x,w}$, where $P_{x,w}$ are the polynomials defined in \cite{KL} and $q^{1/2} = \u$.)
Now let $\mu(x,w) \in \ZZ$ be the coefficient of $\ui$ in $P'_{x,w}$ (resp. $P'_{w,x}$) if $x \leq w$ (resp. $w \leq x$).
Then the right regular representation in terms of the canonical bases of $\H(W)$ takes the following simple forms:
\begin{equation}\label{e prime C on prime canbas}
\C_w \C_s =
\left\{\begin{array}{ll} [2] \C_w & \text{if}\ s \in R(w),\\
\displaystyle\sum_{\substack{\{w' \in W: s \in R(w')\}}} \mu(w',w)\C_{w'} & \text{if}\ s \notin R(w).
\end{array}\right.
\end{equation}
\begin{equation}\label{e C on canbas}
C_w C_s =
\left\{\begin{array}{ll} -[2] C_w & \text{if}\ s \in R(w),\\
\displaystyle\sum_{\substack{\{w' \in  W: s \in R(w')\}}} \mu(w',w)C_{w'} & \text{if}\ s \notin R(w).
\end{array}\right.
\end{equation}

The simplicity and sparsity of this action along with the fact that the right cells of $\Gamma_W$ and $\Gamma'_W$ often give rise to  $\CC(\u) \tsr_\mathbf{A} \H(W)$-irreducibles are among the most amazing and useful properties of canonical bases.
%Let $\alpha: \mathbf{A} \to \mathbf{A}$ is the automorphism determined by $\u \mapsto -\ui$.  The coefficients of the $C$'s in terms of the $T$'s are
%\be C_w = \sum_{x \in W} \alpha(P'_{x,w}) T_x, \ee
\subsection{Induction and restriction of canonical bases} \label{ss Induction and restriction of canonical bases}
It will be important for our applications in \textsection \ref{s Quantum Schur-Weyl duality and canonical bases}--\ref{s Consequences for the canonical bases of M_lambda} that canonical bases behave well under induction and restriction.

Let $J \subseteq S$. Let $\mathbf{A}\Lambda'$ (resp. $\mathbf{A}\Lambda$) be a right cellular subquotient of $\Gamma'_{W_J}$ (resp. $\Gamma_{W_J}$). The next proposition follows from general results about inducing $W$-graphs \cite{HY1, HY2} (see \cite[Propositions 2.6 and 3.4]{B0}). We will only apply this with $\mathbf{A}\Lambda'$ (resp. $\mathbf{A}\Lambda$) the trivial $\H(W)$ representation, which is a cellular submodule (resp. quotient) of $\Gamma'_W$ (resp. $\Gamma_W$).
\begin{proposition}\label{p cell isomorphism induced}
The basis $\Gamma'_{\Lambda', J} := \{ \C_w : w = v\cdot x, \C_v \in \Lambda', x \in \leftexp{J}{W} \} \subseteq \Gamma'_W$ of  $\mathbf{A} \Lambda' \tsr_{\H(W_J)} \H(W)$ can be constructed from the standard basis $\Lambda T' := \{ \C_v \tsr_{\H(W_J)} T_x : \C_v \in \Lambda', x \in \leftexp{J}{W} \}$ in the sense of \cite{Du}:
 $\C_{v x}$ is the unique $\br{\cdot}$-invariant element of  $\ZZ[\ui] \Lambda T'$ congruent to  $\C_v \tsr_{\H(W_J)} T_x \mod \ui \ZZ[\ui] \Lambda T'$.
Hence,  $\mathbf{A} \Gamma'_{\Lambda',J}$  is a right cellular subquotient of $\H(W)$.
 The same statement holds with $\Lambda$ in place of $\Lambda'$,  $\Gamma_W$ in place of $\Gamma'_W$, $C$'s in place of $\C$'s, and $\u$ in place of $\ui$.
\end{proposition}

The next result about restricting canonical bases originated in the work of Barbasch and Vogan on primitive ideals \cite{BV}, and is proven in the generality stated here by Roichman \cite{R} (see also \cite[\textsection3.3]{B0}).
%this proposition was not quite correct as stated, and the should be updated in a couple of other papers ???
\begin{proposition}\label{p restrict Wgraph}
Let $J \subseteq S$ and $E$ be the right $\H(W_J)$-module $\Res_{\H(W_J)} \H(W)$. Then for any $x \in W^J$,  $E_x := \mathbf{A} \{\C_{xv} : v \in W_J \}$ is a cellular subquotient of  $(E,\Gamma'_W)$ and
\be E_x \xrightarrow{\cong} \H(W_J), \C_{xv} \mapsto \C_v \ee
is an isomorphism of right $\H(W_J)$-modules with basis. In particular, any right cell of $(E,\Gamma'_W)$ is isomorphic to one occurring in $\H(W_J)$. The same statement holds for $(E,\Gamma_W)$, with $C$'s replacing $\C$'s.
\end{proposition}
\subsection{Cells in type  $A$} \label{ss cell label conventions C_Q C'_Q}
Let  $\H_r = \H(\S_r)$ be the type $A$ Hecke algebra.

It is well known that $\field \H_r := \field \tsr_\mathbf{A} \H_r$ is semisimple and its irreducibles in bijection with partitions of $r$; let $M_\lambda$ and $M_\lambda^{\mathbf{A}}$ be the $\field\H_r$-irreducible and Specht module of $\H_r$ of shape $\lambda \vdash r$ (hence $M_\lambda \cong \field\tsr_\mathbf{A} M_\lambda^\mathbf{A}$). For any $\field \H_r$-module $N$ and partition $\lambda$ of $r$, let $N[\lambda]$ be the $M_\lambda$-isotypic component of $N$.
%Also set $N[\ld\lambda] = \bigoplus_{\mu \ld\lambda} N[\mu]$ and $N[\gd\lambda] = \bigoplus_{\mu \gd\lambda} N[\mu]$.
Let $s_\lambda^N : N \twoheadrightarrow N[\lambda]$ be the canonical surjection and $i_\lambda^N : N[\lambda] \hookrightarrow N$ the canonical inclusion. Define the projector $p_\lambda^N: N \to N$ by $p_\lambda^N = i_\lambda^N \circ s_\lambda^N$.   We also let $p_\lambda$ denote central idempotent of $\field \H_r$ so that the map $p_\lambda^N$ is given by multiplication by $p_\lambda$.

The work of Kazhdan and Lusztig \cite{KL} shows that the decomposition of $\Gamma_{\S_r}$ into right cells is
$\Gamma_{\S_r} = \bigsqcup_{P \in \text{SYT}^r} \Gamma_P$, where $\Gamma_P := \{C_w : P(w) = P\}$.  Moreover, the right cells $\{ \Gamma_P : \sh(P) = \lambda\}$ are all isomorphic, and, denoting any of these cells by $\Gamma_\lambda$, $\mathbf{A}\Gamma_\lambda \cong M_\lambda^\mathbf{A}$. Similarly, the decomposition of $\Gamma'_{\S_r}$ into right cells is
$\Gamma'_{\S_r} = \bigsqcup_{P \in \text{SYT}^r} \Gamma'_P$, where $\Gamma'_P := \{\C_w : \transpose{P(w)} = P\}$.  Moreover, the right cells $\{ \Gamma'_P : \sh(P) = \lambda\}$ are all isomorphic, and, denoting any of these cells by $\Gamma'_\lambda$, $\mathbf{A}\Gamma'_\lambda \cong M_\lambda^\mathbf{A}$.
A combinatorial discussion of left cells in type $A$ is given in \cite[\textsection 4]{B0}.

We refer to the basis $\Gamma_\lambda$ of $M_\lambda^\mathbf{A}$ as the \emph{upper canonical basis of $M_\lambda$} and denote it by $\{ C_Q : Q \in \text{SYT}(\lambda) \}$, where $C_Q$ corresponds to $C_w$ for any (every) $w \in \S_r$ with recording tableau $Q$. Similarly, the basis $\Gamma'_\lambda$ of $M_\lambda^\mathbf{A}$ is the \emph{lower canonical basis of $M_\lambda$}, denoted $\{ \C_Q : Q \in \text{SYT}(\lambda) \}$, where $\C_Q$ corresponds to $\C_w$ for any (every) $w \in \S_r$ with recording tableau $\transpose{Q}$. Note that with these labels the action of $C_s$ on the upper canonical basis of $M_\lambda$ is similar to \eqref{e C on canbas}, with  $\mu(Q',Q):= \mu(w',w)$ for any  $w',w$ such that  $P(w')=P(w)$, $Q' = Q(w'),Q = Q(w)$, and right descent sets
\be
R(C_Q) = \{ s_i : i + 1 \text{ is strictly to the south of $i$ in $Q$}\}.
\ee
Similarly, the action of $\C_s$ on $\{ \C_Q : Q \in \text{SYT}(\lambda) \}$ is similar to \eqref{e prime C on prime canbas}, with  $\mu(Q',Q):= \mu(w',w)$ for any  $w',w$ such that  $\transpose{P(w')}=\transpose{P(w)}$, $Q' = \transpose{Q(w')},Q = \transpose{Q(w)}$, and right descent sets
\be
R(\C_Q) = \{ s_i : i + 1 \text{ is strictly to the east of $i$ in $Q$}\}.
\ee

\begin{example} \label{ex lambda31}
The integers $\mu(Q',Q)$ for both the upper and lower canonical basis of $M_{(3,1)}$ are given by the following graph ($\mu$ is 1 if the edge is present and 0 otherwise)
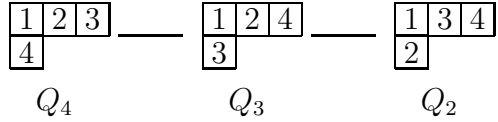
\[
\xymatrix@R=0cm{
\ \tableau{1&2&3\\4} \ar @{-} [r] & \ \tableau{1&2&4\\3} \ar @{-} [r] & \ \tableau{1&3&4\\2} \\
Q_4 & Q_3 & Q_2
}
\]
The right action of the $\C_s$ on $(\C_{Q_4}, \C_{Q_3}, \C_{Q_2})$ is given by (the columns of the matrices are $\C_{Q} \C_s$ in terms of the $\C_Q$-basis)
\[
\C_{s_1} \mapsto \begin{pmatrix}
[2] & 0 & 0 \\
0 & [2] & 1 \\
0 & 0 & 0
\end{pmatrix} \quad
\C_{s_2} \mapsto \begin{pmatrix}
[2] & 1 & 0 \\
0 & 0 & 0 \\
0 & 1 & [2]
\end{pmatrix} \quad
\C_{s_3} \mapsto \begin{pmatrix}
0 & 0 & 0 \\
1 & [2] & 0 \\
0 & 0 & [2]
\end{pmatrix}
\]
The right action of the $\C_s$ on $(C_{Q_4}, C_{Q_3}, C_{Q_2})$ is given by
\[
\C_{s_1} \mapsto \begin{pmatrix}
[2] & 0 & 0 \\
0 & [2] & 0 \\
0 & 1 & 0
\end{pmatrix} \quad
\C_{s_2} \mapsto \begin{pmatrix}
[2] & 0 & 0 \\
1 & 0 & 1 \\
0 & 0 & [2]
\end{pmatrix} \quad
\C_{s_3} \mapsto \begin{pmatrix}
0 & 1 & 0 \\
0 & [2] & 0 \\
0 & 0 & [2]
\end{pmatrix}
\]
Note that the matrix corresponding to the right action of $\C_s$ on the $\C_Q$-basis is transpose to the action in the $C_Q$-basis. This is true in general---it is a consequence of Proposition \ref{p dual bases Mlambda} or of \cite[Corollary 3.2]{KL}.
\end{example}

%The preorder $\klo{E}$ induces a partial order on the cells of $E$, which is also denoted $\klo{E}$. This seems to be quite difficult to compute completely; it is not even known for the $\S_n$-graph $\Gamma_{\S_n}$. We will see some results that help determine $\klo{E}$ throughout the paper.
The partial orders  $\klo{\Gamma_{\S_r}}$ and  $\klo{\Gamma'_{\S_r}}$ are not well understood, but there is the following deep result which gives us some understanding.  The result follows from Lusztig's $a$-invariant and the nonnegativity of the structure constants of the $\C_w$ due to Beilinson-Bernstein-Deligne-Gabber \cite[\textsection 5-6]{L2} and results of \cite{BV} and \cite{Joseph} on primitive ideals of $\Uq$ (see the appendix of \cite{Himmanant}).
\begin{theorem}\label{t right cell partial order}
The partial order on the right cells of $\Gamma_{\S_r}$ and $\Gamma'_{\S_r}$ is constrained by dominance order: if $\Gamma_{P'} \kloneq{\Gamma_{\S_r}} \Gamma_{P}$, then $\sh(P') \ldneq \sh(P)$; if $\Gamma'_{P'} \kloneq{\Gamma'_{\S_r}} \Gamma'_{P}$, then $\sh(P') \gdneq \sh(P)$.
\end{theorem}
%I might not need this result after all ???
% I think that the part that follows from the $a$-invariant is also not needed, though this is somewhat surprising to me
%\begin{proof}
%it might be possible to prove this using quantum Schur-Weyl duality as below:
%By the results in \cite[Chapter 27]{LBook} (see the discussion at the end of \textsection\ref{ss global canonical bases}), the partial order on the $\Uq$-cells of $(\bT, B')$ is given by dominance order. Thus for $\mathbf{k}$ such that $\sh(\mathbf{k}^{\dagger}) = \lambda$, $\pi_\lambda^\bT(\cp_\mathbf{k}) = p_\lambda^\bT(\cp_\mathbf{k})$ belongs to $\bT[\gd \lambda]$. This means that the minimal $\H_r$-cellular submodule containing $\cp_\mathbf{k}$ is contained in $\bigsqcup_{\stackrel{T \in \text{SSYT}_{[n]}(\mu),}{\mu \gd \lambda, \, T \text{ has the same content as } \mathbf{k}}} \Gamma'_T$
%\end{proof}
\section{The quantized enveloping algebra and crystal bases}
\label{s Crystal bases of the quantized enveloping algebra}
We recall the definition of the quantized enveloping algebra $\Uq = U_q(\gl_n)$ following \cite{Kas1,HK}. We then briefly recall the construction of global crystal bases in the sense of \cite{Kas1,Kas2} and of the similar notion of based modules of \cite{LBook}.

\subsection{Definition of $\Uq = U_q(\gl_n)$ and basic properties}
The \emph{quantized universal enveloping algebra} $\Uq$ is the associative $\field$-algebra generated by
$q^h, h \in \dual{X}$ (set $K_i = q^{\dual{\epsilon}_{i}-\dual{\epsilon}_{i+1}}$) and $E_i, F_i, i \in [n - 1]$ with relations
\be
\begin{array}{ll}
q^0 = 1, & q^hq^{h'} = q^{h + h'},   \\
q^hE_iq^{-h} = \u^{\langle \alpha_i, h \rangle}E_i, & q^hF_iq^{-h} = \u^{-\langle \alpha_i, h \rangle}F_i, \\
E_iF_j - F_jE_i = \delta_{i,j} \frac{K_i - K_i^{-1}}{\u-\ui}, & \\
E_iE_j - E_jE_i = F_iF_j - F_jF_i = 0 & \text{for } |i-j| > 1, \\
E_i^2E_j-  [2]E_iE_jE_i + E_jE^2_i = 0 & \text{for } |i-j| = 1, \\
F^2_iF_j - [2]F_iF_jF_i + F_jF^2_i = 0 & \text{for } |i-j| = 1. \\
\end{array}
\ee
\begin{remark}
Our notation is related to that of Kashiwara's and Brundan's \cite{Brundan} by $\u = q$.  We use $\u$ instead of $q$ because on the Hecke algebra side, our $\u$ is what is usually $q^{1/2}$.
%Note that $K_i$ is $t_i$ in the notation of Kashiwara \cite{Kas1, Kas2}.
\end{remark}

%We will need the following antiautomorphism  and automorphism of $\Uq$.
The \emph{bar-involution}, $\br{\cdot}: \Uq \to \Uq$ is the  $\QQ$-linear automorphism extending the involution $\br{\cdot}$ on $\field$ and satisfying
\be \br{q^{h}} = q^{-h}, \ \br{E_i} = E_i, \ \br{F_i} = F_i. \ee
Let $\varphi : \Uq \to \Uq$ be the algebra antiautomorphism determined by
\be
\varphi(E_i) = F_i, \quad \varphi(F_i) = E_i, \quad \varphi(K_i) = K_i.
\ee
%The bar-involution, $\br{\cdot}$: $V_\lambda \to V_\lambda$, is the additive map extending the involution $\br{\cdot}$ on $\field$ and satisfying
%\[ \br{v_\lambda} = v_\lambda,\]
%where $v_\lambda$ is a highest weight vector of $V_\lambda$.

The algebra $\Uq$ is a Hopf algebra with coproduct $\Delta$ given by
\be \label{e U_q coproduct}
\Delta(q^h) = q^h \tsr q^h, \ \ \Delta(E_i) = E_i \tsr K_i^{-1} + 1 \tsr E_i, \ \ \Delta(F_i) = F_i \tsr 1 + K_i \tsr F_i.
\ee
This is the same as the coproduct used in \cite{Brundan, Kas2, HK}, and it differs from the coproduct $\tilde{\Delta}$ of \cite{LBook} by $(\varphi \tsr \varphi) \circ \Delta \circ \varphi$.
%??useful coproduct in FKK differs from both of these

The \emph{weight space}  $N^\zeta$ of a $\Uq$-module  $N$ for the weight $\zeta \in X$ is the $\field$-vector space $\{x \in N : q^h x = \u^{\langle \zeta, h \rangle} x\}$.
Let $\Oint$ be as in \cite[Chapter 7]{HK}, the category of finite-dimensional $\Uq$-modules such that the weight of any non-zero weight space belongs to $\ZZ^n_{\geq 0} \subseteq X$. It is semisimple, the simple objects being the highest weight modules $V_\lambda$ for partitions $\lambda$.

For any object $N$ of $\Oint$ and partition $\lambda$, let $N[\lambda]$ be the $V_\lambda$-isotypic component of $N$. Set $N[\ld\lambda] = \bigoplus_{\mu \ld\lambda} N[\mu]$, $N[\ldneq\lambda] = \bigoplus_{\mu \ldneq\lambda} N[\mu]$, $N[\gd\lambda] = \bigoplus_{\mu \gd\lambda} N[\mu]$, and $N[\gdneq\lambda] = \bigoplus_{\mu \gdneq\lambda} N[\mu]$. Let $\varsigma_\lambda^N : N \twoheadrightarrow N[\lambda]$ be the canonical surjection and $\iota_\lambda^N : N[\lambda] \hookrightarrow N$ the canonical inclusion. Define the projector $\pi_\lambda^N: N \to N$ by $\pi_\lambda^N = \iota_\lambda^N \circ \varsigma_\lambda^N$.

\subsection{Crystal bases}
%Let $\L_0(\lambda)$ be the  $\field_0$-submodule of $V_\lambda$ and $\B'(\lambda)$ the $\QQ$-basis of $\L_0(\lambda) / \u \L_0(\lambda)$ so that  $(\L_0(\lambda),\B'(\lambda))$ is the lower  crystal basis  of $V_\lambda$ at  $\u = 0$ in the sense of \cite{Kas2}. Similarly, let $\L_\infty(\lambda)$ be the  $\field_\infty$-submodule of $V_\lambda$ and $\B(\lambda)$ the $\QQ$-basis of $\L_\infty(\lambda) / \ui \L_\infty(\lambda)$ so that  $(\L_\infty(\lambda),\B(\lambda))$ is the upper  crystal basis  of $V_\lambda$ at $\u = \infty$.

A lower crystal basis at  $\u = 0$ of an object $N$ of $\Oint$ is a pair  $(\L_0(N),\B')$, where $\L_0(N)$ is a $\field_0$-submodule of $N$ and $\B'$ is a $\QQ$-basis of $\L_0(N) / \u \L_0(N)$ which satisfy a certain compatibility with the Kashiwara operators $\crystall{E_i}, \crystall{F_i}$; an upper crystal basis at  $\u = \infty$ of $N$ is a  pair  $(\L_\infty(N),\B)$, where $\L_\infty(N)$ is a $\field_\infty$-submodule of $N$ and $\B$ is a $\QQ$-basis of $\L_\infty(N) / \ui \L_\infty(N)$ which satisfy a certain compatibility with the Kashiwara operators $\crystalu{E_i}, \crystalu{F_i}$ (see \cite[\textsection3.1]{Kas2}).

Kashiwara \cite{Kas2} gives a fairly explicit construction of a lower (resp. upper) crystal basis of $V_\lambda$, which we denote by $(\L_0(\lambda),\B'(\lambda))$ (resp. $(\L_\infty(\lambda),\B(\lambda))$). The basis  $\B'(\lambda)$ (resp.  $\B(\lambda)$) is naturally labeled by $\text{SSYT}_{[n]}(\lambda)$ and we let $b'_P$ (resp. $b_P$) denote the basis element corresponding to $P \in \text{SSYT}_{[n]}(\lambda)$ (see, for instance, \cite[Chapter 7]{HK}).
A fundamental result of \cite{Kas1,Kas2} is that a lower (resp. upper) crystal basis is always isomorphic to a direct sum  $\bigoplus_j (\L_0(\lambda^j),\B'(\lambda^j))$ (resp. $\bigoplus_j (\L_\infty(\lambda^j),\B(\lambda^j))$).

%We probably don't need these theorems now that we state the result of Lusztig for tensor products of base modules
%\begin{theorem}[Kashiwara (see {\cite[Theorem 4.4.1]{HK}})]
%\label{t theorem 4.4.1 HK}
%Let $N_j$ be a $\Uq$-module in the category $\O_{\text{int}}^q$ and let $(\L_j, \B_j)$ be a crystal basis of $N_j \ (j=1,2)$. Set $\L = \L_1 \tsr_{\field_\infty} \L_2$ and $\B = \B_1 \times \B_2$. Then $(\L, \B)$ is a crystal basis of $N_1 \tsr_{\field} N_2$.
%\end{theorem}
%
%\begin{theorem}[Kashiwara (see {\cite[Theorem 5.2.1]{HK}})]
%\label{t theorem 5.2.1 HK}
%Let $N$ be a $\Uq$-module in the category $\O_\text{int}^q$ with a crystal basis $(\L, \B)$. If $N \cong \bigoplus_{\lambda \in Y_+} V(\lambda)^{\oplus a_\lambda} \ (a_\lambda \in \ZZ_{\geq 0})$, then there exists an isomorphism of crystal bases
%\[ \Psi : (\L, \B) \xrightarrow{\cong} \left( \bigoplus_{\lambda \in Y_+} \L(\lambda)^{\oplus a_\lambda}, \bigsqcup_{\lambda \in Y_+} \B(\lambda)^{\oplus a_\lambda} \right). \]
%\end{theorem}

%We let \emph{upper crystal basis} and \emph{lower crystal basis} denote the upper crystal base and lower crystal base of \cite{Kas1, Kas2}. An upper crystal basis of a $\Uq$-module $N$ is a pair $(\L,\B)$, where $\L$ is a $\field_\infty$-submodule of $N$ such that $\field \tsr_{\field_\infty} \L = N$ and $\B$ is a $\QQ$-basis of $\L / \ui \L$ satisfying several compatibility conditions for the action of $\Uq$. A lower crystal basis is also a pair $(\L, \B)$ satisfying different compatibility conditions.

\subsection{Global crystal bases}
\label{ss global canonical bases}
We next define lower based modules and upper based modules, where a lower based module is a based module in the sense of \cite[Chapter 27]{LBook} adapted to our coproduct.

The \emph{$\mathbf{A}$-form $\Uq_\mathbf{A}$ of $\Uq$} is the $\mathbf{A}$-subalgebra of $\Uq$ generated by $ \frac{E_i^m}{[m]!},  \frac{F_i^m}{[m]!}, q^h, \genfrac{\{}{\}}{0pt}{}{q^{h}}{m}$ for $i \in [n-1],\ m \in \ZZ_{\geq 0}$, and $h \in \dual{X}$, where
\[\genfrac{\{}{\}}{0pt}{}{x}{m} := \prod^m_{k=1} \frac{\u^{1-k}x - \u^{k-1}x^{-1}}{\u^k-\u^{ -k}}. \]
We also define the \emph{$\QQA$-form $\Uq_\QQ$ of  $\Uq$} to be $\QQ \tsr_\ZZ \Uq_{\mathbf{A}}$.

\begin{definition}
A \emph{lower based module} is a pair $(N,B)$, where $N$ is an object of $\Oint$
%, finite dimensional over $\field$,
%incorporated in oint definition --no longer needed
and $B$ is a $\field$-basis of $N$ such that
\begin{list}{(\alph{ctr})} {\usecounter{ctr} \setlength{\itemsep}{1pt} \setlength{\topsep}{2pt}}
\item $B \cap N^{\zeta}$ is a basis of $N^\zeta$, for any $\zeta \in X$;
\item Define $N_\mathbf{A} := \mathbf{A} B$. The $\QQA$-submodule  $\QQ \tsr_\ZZ N_\mathbf{A}$ of $N$ is stable under  $\Uq_\QQ$;
\item the $\QQ$-linear involution $\br{\cdot} : N \to N$ defined by $\br{ab} = \br{a}b$ for all $a \in \field$ and all $b \in B$ intertwines the $\br{\cdot}$-involution of $\Uq$, i.e. $\br{fn} = \br{f}\br{n}$ for all $f \in \Uq, n \in N$;
\item Set $\L_0(N) = \field_0 B$ and let $\B$ denote the image of $B$ in $\L_0(N)/\u \L_0(N)$. Then $(\L_0(N), \B)$ is a lower crystal basis of $N$ at  $\u = 0$.
\end{list}
\end{definition}

\begin{definition}
An \emph{upper based module} is the same as a lower based module except with condition (d) replaced by
\begin{list}{\emph{(\alph{ctr})}} {\usecounter{ctr} \setlength{\itemsep}{1pt} \setlength{\topsep}{2pt}}
\item[] Set $\L_\infty(N) = \field_\infty B$ and let $\B$ denote the image of $B$ in $\L_\infty(N)/\ui \L_\infty(N)$. Then $(\L_\infty(N), \B)$ is an upper crystal basis of $N$ at $\u = \infty$.
\end{list}
\end{definition}
The  \emph{$\br{\cdot}$-involution} of the lower (resp.  upper) based module is the involution on $N$ defined in (c). The \emph{balanced triple} of a lower (resp.  upper) based module is $(\QQA B, \field_0B, \field_\infty B)$.

\begin{remark}
For simplicity and to be consistent with the treatment of upper global crystal bases in \cite{Kas2}, we have used the $\QQA$-form $\Uq_\QQ$ from \cite{Kas2} rather than the $\mathbf{A}$-form of $\dot{\Uq}$ defined in \cite{LBook}.
\end{remark}

\begin{remark}\label{r balanced triple}
In the language of Kashiwara \cite{Kas2}, the basis  $B$ in the definitions above is a lower or upper \emph{global crystal basis}.  Kashiwara defines a triple $(N_\QQ,\L_0,\L_\infty)$ to be \emph{balanced} if  the canonical surjection $N_\QQ \cap \L_0 \cap \L_\infty \to \L_0 / \u \L_0$ is an isomorphism, where  $N$ is any $\field$-vector space, and  $N_\QQ$,  $\L_0$,  $\L_\infty$ are any $\QQA$-submodule, $\field_0$-submodule, and $\field_\infty$-submodule  of  $N$, respectively.  To define global lower crystal bases, Kashiwara first defines a balanced triple  $(\QQ \tsr_\ZZ N_\mathbf{A}, \L_0(N),\br{\L_0(N)}) $ and a basis  $\B \subseteq \L_0 / \u \L_0$ and then defines $B$ to be the inverse image of  $\B$ under the isomorphism
\[ \QQ \tsr_\ZZ N_\mathbf{A} \cap \L_0(N) \cap \br{\L_0(N)}  \xrightarrow{\cong} \L_0 / \u \L_0.\]
Global upper canonical bases are defined similarly.
\end{remark}

Let $\eta_\lambda$ be a highest weight vector of $V_\lambda$.  The $\br{\cdot}$-involution on $V_\lambda$ is defined by setting $\br{\eta_\lambda} = \eta_\lambda$ and requiring that it intertwines the $\br{\cdot}$-involution of $\Uq$. The  $\QQA$-forms of  $V_\lambda$ of \cite{Kas2} are denoted $V^{\QQ \text{ low}}_\lambda$ and  $V^{\QQ \text{ up}}_\lambda$;  $V^{\QQ \text{ low}}_\lambda$ is defined to be  $\Uq_\QQ \eta_\lambda$  and $V^{\QQ \text{ up}}_\lambda$ is defined by dualizing  $V^{\QQ \text{ low}}_\lambda$ by a symmetric form on  $V_\lambda$.  We can now state the fundamental result about the existence of global crystal bases and based modules for $V_\lambda$.
\begin{theorem}[Kashiwara \cite{Kas1,Kas2}]
\label{t theorem 6.2.2 HK}
% Let $V_\lambda$ be the irreducible highest weight $\Uq$-module with the highest weight $\lambda \in Y_+$ and let $(\L(\lambda), \B(\lambda))$ be the crystal basis of $V_\lambda$.
%There exists a unique $\br{\cdot}$-invariant $\mathbf{A}$-basis $G(\B(\lambda))$ of $V_\lambda^{\mathbf{A}}$ parameterized by $\B(\lambda)$ such that $G(b) \equiv b \mod \ui \L(\lambda)$.
\
\begin{list}{\emph{(\roman{ctr})}} {\usecounter{ctr} \setlength{\itemsep}{1pt} \setlength{\topsep}{2pt}}
\item The triple $(V^{\QQ \text{ low}}_\lambda, \L_0(\lambda), \br{\L_0(\lambda)})$ is balanced.  Then, letting $G'_\lambda$ be the inverse of the canonical isomorphism
\[ V^{\QQ \text{ low}}_\lambda \cap \L_0(\lambda) \cap \br{\L_0(\lambda)}  \xrightarrow{\cong} \L_0(\lambda) / \u \L_0(\lambda),\]
 $B'(\lambda) := G'_\lambda(\B'(\lambda))$ is the \emph{lower global crystal basis of  $V_\lambda$} and $(V_\lambda,B'(\lambda))$ is a lower based module.
\item The triple $(V^{\QQ \text{ up}}_\lambda, \br{\L_\infty(\lambda)}, \L_\infty(\lambda))$ is balanced.  Then, letting $G_\lambda$ be the inverse of the canonical isomorphism
\[ V^{\QQ \text{ up}}_\lambda \cap \br{\L_\infty(\lambda)} \cap \L_\infty(\lambda)  \xrightarrow{\cong} \L_\infty(\lambda) / \ui \L_\infty(\lambda),\]
 $B(\lambda) := G_\lambda(\B(\lambda))$ is the \emph{upper global crystal basis of  $V_\lambda$} and $(V_\lambda,B(\lambda))$ is an upper based module.
 \end{list}
\end{theorem}
Note that Kashiwara proves that the triples are balanced and the conclusions about based modules follow easily (see \cite[27.1.4]{LBook} or {\cite[Theorem 6.2.2]{HK}}).  We may now define integral forms  $V^{\mathbf{A} \text{ low}}_\lambda := \mathbf{A}B'(\lambda)$ and $V^{\mathbf{A} \text{ up}}_\lambda := \mathbf{A}B(\lambda)$ of $V_\lambda$.

We wish to make use of some of the facts established about lower based modules in \cite[Chapter 27]{LBook} and their corresponding statements for upper based modules. It is shown in \cite[Chapter 27]{LBook} that if $(N,B)$ is a lower based module, then so are $(N[\gd \lambda], B[\gd \lambda])$ and  $(N[\gd \lambda] /N[\gdneq \lambda], B[\gd \lambda] - B[\gdneq \lambda])$, where $B[\gd \lambda] = N[\gd \lambda] \cap B$, etc. Moreover, this last based module is isomorphic to a direct sum of copies of $V_\lambda$ with their lower global canonical bases.  The analogous statements for upper based modules are true with $\gd$ replaced by $\ld$ and are shown in \cite[\textsection 5.2]{Kas2}.

%I'm pretty sure all of these facts carry over for upper based modules, including tensor products
\subsection{Tensor products of based modules}
\label{ss Tensor products of based modules}
Let $(N,B), (N',B')$ be lower (resp.  upper) based modules.  There is a basis  $B \diamond B'$ (resp.  $B \heart B'$) which makes $N \tsr N'$ into a lower (resp.  upper) based module.  However, first, we need an involution on  $N \tsr N'$ that intertwines the $\br{\cdot}$-involution on $\Uq$.  This definition is not obvious and requires Lusztig's quasi-$R$-matrix, but adapted to our coproduct as in \cite{Brundan}: let $\Theta = (\varphi \tsr \varphi)(\tilde{\Theta}^{-1})$ where $\tilde{\Theta}$ is exactly Lusztig's quasi-$\R$-matrix from \cite[4.1.2]{LBook}. It is an element of a certain completion $(\Uq \tsr \Uq)^{\wedge}$ of the algebra $\Uq \tsr \Uq$.
Then the involution $\br{\cdot}:N \tsr N' \to N \tsr N'$ is defined by $\br{n \tsr n'} = \Theta(\br{n} \tsr \br{n'})$. (This involution is denoted $\Psi$ in \cite{LBook}.)
%The next theorem is essentially \cite[Theorem 27.3.2]{LBook}.
\begin{theorem}[Lusztig {\cite[Theorem 27.3.2]{LBook}}] \label{t tsr product lower canbas}
Maintain the notation above with $(N,B), (N', B')$ lower based modules and set $(N \tsr N')_{\ZZ[\u]} = \ZZ[\u] B \tsr B'$. For any $(b,b') \in B \times B'$, there is a unique element $b \diamond b' \in (N \tsr N')_{\ZZ[\u]}$ such that $\br{b\diamond b'} = b \diamond b'$ and $(b\diamond b')-b\tsr b' \in \u (N\tsr N')_{\ZZ[\u]}$.

Set $B \diamond B' = \{b\diamond b' : b \in B, b' \in B'\}$. Then the pair $(N \tsr N', B\diamond B')$ is a lower based module.
\end{theorem}

There is a similar theorem for upper based modules, as the proof of Theorem \ref{t tsr product lower canbas} adapts easily.  This is discussed in \cite{FK} in the $n=2$ case, and we use the notation $\heart$ for this product as is done there.
\begin{theorem} \label{t tsr product upper canbas}
Maintain the notation above with $(N,B), (N', B')$ upper based modules and set $(N \tsr N')_{\ZZ[\ui]} = \ZZ[\ui] B \tsr B'$. For any $(b,b') \in B \times B'$, there is a unique element $b \heart b' \in (N \tsr N')_{\ZZ[\ui]}$ such that $\br{b\heart b'} = b \heart b'$ and $(b\heart b')-b\tsr b' \in \ui (N\tsr N')_{\ZZ[\ui]}$.

Set $B \heart B' = \{b\heart b' : b \in B, b' \in B'\}$. Then the pair $(N \tsr N', B\heart B')$ is an upper based module.
\end{theorem}
Moreover, the products $\diamond$ and $\heart$ are associative (\cite[27.3.6]{LBook}).

\section{Quantum Schur-Weyl duality and canonical bases}
\label{s Quantum Schur-Weyl duality and canonical bases}
Write $V$ for the natural representation  $V_{\epsilon_1}$ of $\Uq$.
The action of  $\Uq$ on the weight basis $v_1,\ldots,v_n$ of  $V$ is given by $q^{\dual{\epsilon}_{i}} v_j = \u^{\delta_{ij}} v_j$, $F_i v_i = v_{i+1}$, $F_i v_j = 0$ for $i \neq j$, and $E_i v_{i+1} = v_i$, $E_i v_j = 0$ for $j \neq i+1$.

We recall the commuting actions of $\Uq$ and $\H_r$ on $\bT := V^{\tsr r}$ as described in \cite{Jimbo,GL, Ram,FKK, Brundan} and give several characterizations of the lower and upper canonical basis of $\bT$; we closely follow \cite{Brundan} and are consistent with its conventions.

\subsection{Commuting actions on $\bT = V^{\tsr r}$}
The action of $\Uq$ on $\bT$ is determined by the coproduct $\Delta$  \eqref{e U_q coproduct}.
The commuting action of $\H_r$ on $\bT$ comes from a $\Uq$-isomorphism  $\R_{V,V} : V \tsr V \to V \tsr V$ determined by the universal $R$-matrix; this isomorphism can also be defined using the quasi-$\R$-matrix \cite[32.1.5]{LBook} (see also \cite[\textsection 3]{Brundan}).  The  $\H_r$ action is given explicitly on generators as follows: for a word $\mathbf{k} = k_1\dots k_r \in [n]^r$, let $\bv_{\mathbf{k}} = v_{k_1} \tsr v_{k_2} \tsr \dots \tsr v_{k_r}$ be the corresponding tensor monomial. Recall from  \textsection\ref{ss type A combinatorics preliminaries} the right action of  $\S_r$ on words of length  $r$. Then
\be
\label{e T inverse_i act on V}
\bv_\mathbf{k} T_i^{-1} =
\begin{cases}
\bv_{\mathbf{k} \, s_i} & \text{ if } k_i < k_{i+1}, \\
\ui \bv_\mathbf{k} & \text{ if } k_i = k_{i + 1}, \\
(\ui - \u) \bv_\mathbf{k} + \bv_{\mathbf{k} \, s_i} & \text{ if } k_i > k_{i+1}.
\end{cases}
\ee

\begin{remark}
This convention for the action of $\H_r$ on $\bT$ is consistent with that in \cite{Brundan, Ram, GCT4, BMSGCT4} and \cite[Proposition 2.1$'$]{FKK}, but not with that in \cite{GL} and \cite[Proposition 2.1]{FKK}. Note that  $\bv_\mathbf{k}, T_i^{-1}$ are denoted $M_\alpha, H_i$ respectively in \cite{Brundan}.
%The action of $\Uq$ on $\bT$ given by the comultiplication above is consistent with that in \cite{HK}.
\end{remark}

We can now state the beautiful quantum version of Schur-Weyl duality, originally due to Jimbo \cite{Jimbo}.
\begin{theorem}\label{c Schur-Weyl duality basic}
As a $(\Uq, \field\H_r)$-bimodule, $\bT$ decomposes into irreducibles as
\[
\bT \cong \bigoplus_{\lambda \vdash_n r}  V_\lambda \tsr M_\lambda.
\]
\end{theorem}

As an $\H_r$-module, $\bT$ decomposes into a direct sum of weight spaces: $\bT \cong \bigoplus_{\zeta \in X} \bT^\zeta$.  The weight space $\bT^\zeta$ is the $\field$-vector space spanned by $\bv_{\mathbf{k}}$ such that $\mathbf{k}$ has content $\zeta$.
Let $\epsilon_+ := M_{(r)}^\mathbf{A}$ be the trivial $\H_r$-module, i.e. the one-dimensional module identified with the map $\H_r \to \mathbf{A}$, $T_i \mapsto \u$. It is not difficult to prove using \eqref{e T inverse_i act on V} (see \cite[\textsection 4]{Brundan})

\begin{proposition}
\label{p weight space equals induced}
The map $\bT_\mathbf{A}^\zeta \to \epsilon_+ \tsr_{\H_{J_\zeta}} \H_r$ given by $\bv_\mathbf{k} \mapsto \epsilon_+ \tsr_{\H_{J_\zeta}} \br{T}_{d(\mathbf{k})}$ is an isomorphism of right $\H_r$-modules.
\end{proposition}
Here $\bT_\mathbf{A}$ is the integral form of $\bT$, defined below.
%\begin{proof}
%In view of (\ref{e T inverse_k act on V}), it is immediate that \[ \bv_{\mathbf{k}} \mapsto \epsilon_+ \bt \br{T_w} \] defines a right  $\H_r$-module homomorphism, where $w$ is the unique element of $\leftexp{J_\zeta}{\S_r}$ such that  $(\sort(\mathbf{k}))w = \mathbf{k}$. This homomorphism is an isomorphism since $\{\bv_{\mathbf{k}}: \mathbf{k} \in T_\zeta \}$ and  $\{\epsilon_+ \bt \br{T_w}: w \in \leftexp{J_\zeta}{\S_r} \}$ are bases of $\bT^\zeta$ and $\epsilon_+ \tsr_{\H_{J_\zeta}} \H_r$ respectively.
%\end{proof}
\subsection{Lower canonical basis of $\bT$}
We now apply the general theory of \textsection\ref{s Crystal bases of the quantized enveloping algebra} to construct global crystal bases of $\bT$.  Recall from  \textsection\ref{ss Tensor products of based modules} that there is a $\br{\cdot}$-involution on  $\bT$ defined using the quasi-$\R$-matrix.
The $\br{\cdot}$-involution on $\H_r$ intertwines that of $\bT$, i.e.
\be \label{e br involution on H intertwines}
\br{v h} = \br{v} \ \br{h}, \text{ for any } v \in \bT, \ h \in \H_r.
\ee
This follows easily from the identity $\tilde{\Theta}^{-1} = \br{\tilde{\Theta}}$ from \cite{LBook}; see \cite{Brundan}.
%\begin{proposition} \label{p br involution on H intertwines}
%The $\br{\cdot}$-involution on $\H_r$ intertwines that of $\bT$, i.e.
%\[ \br{v h} = \br{v} \ \br{h}, \text{ for any } v \in \bT, \ h \in \H. \]
%\end{proposition}

Let $V_\mathbf{A} = \mathbf{A}\{v_i: i \in [n]\}$, which is the same as the integral forms $V^{\mathbf{A} \text{ low}}_{\epsilon_1} = V^{\mathbf{A} \text{ up}}_{\epsilon_1}$ from  \textsection\ref{ss global canonical bases}.
By Theorem \ref{t tsr product lower canbas} and associativity of the $\diamond$ product, $(\bT, B')$ is a lower based module with balanced triple $(\QQ \tsr_\ZZ \bT_\mathbf{A}, \L_0, \br{\L_0})$, where
\be
\begin{array}{ccl}
\L_0 &:=& \L_0(\epsilon_1) \tsr_{\field_0} \dots \tsr_{\field_0} \L_0(\epsilon_1), \\
\B' &:=& \B'(\epsilon_1) \times \dots \times \B'(\epsilon_1) \subseteq \L_0 / \u\L_0, \\
\bT_{\ZZ[\u]} &:=& \ZZ[\u] \{\bv_\mathbf{k} : \mathbf{k} \in [n]^r\}, \\
\bT_\mathbf{A} &:=& V_\mathbf{A} \tsr_\mathbf{A} \dots \tsr_\mathbf{A} V_\mathbf{A} = \mathbf{A} \tsr_\ZZ \bT_{\ZZ[\u]}, \\
B' &:=& B'(\epsilon_1) \diamond \dots \diamond B'(\epsilon_1).
\end{array}
\ee
We call $B'$ the \emph{lower canonical basis} of $\bT$ and, for each  $\mathbf{k} \in [n]^r$, we write $\cp_\mathbf{k}$ for the element $v_{k_1} \diamond \dots \diamond v_{k_r} \in B'$ and $b'_\mathbf{k} \in \B'$ for its image in $\L_0 / \u \L_0$.
 %$(\L_0, \B')$ is a lower crystal basis of $\bT$ at $\u = 0$.
Figure \ref{f c_111 example lower} gives the lower canonical basis in terms of the monomial basis for $r = 3, n =2$.

\begin{figure}[H]
\begin{tikzpicture}[xscale = 5,yscale = 3]
\tikzstyle{vertex}=[inner sep=0pt, outer sep=3pt, fill = white]
\tikzstyle{edge} = [draw, thick, ->,black]
\tikzstyle{LabelStyleH} = [text=black, anchor=south]
\tikzstyle{LabelStyleV} = [text=black, anchor=east]
    \node[vertex] (111) at (2.8,0) {$\cp_{111} = \bv_{111} $};

    \node[vertex] (112) at (2,-1) {$\cp_{112} = \bv_{112}$};
    \node[vertex] (121) at (2, 0) {$\cp_{121} = \bv_{121} + \u \bv_{112} $};
    \node[vertex] (211) at (2, 1) {$\cp_{211}= \atop \bv_{211} + \u \bv_{121} + \u^2 \bv_{112}$};

    \node[vertex] (122) at (1, -1) {$\cp_{122}= \bv_{122} $};
    \node[vertex] (212) at (1 , 0) {$\cp_{212}  = \bv_{212} + \u \bv_{122} $};
    \node[vertex] (221) at (1, 1) {$\cp_{221}= \atop \bv_{221} + \u \bv_{212} + \u^2 \bv_{122}$};

    \node[vertex] (222) at (0, 0) {$\cp_{222}= \bv_{222} $};

\draw[edge] (111) to node[LabelStyleH]{$ $} (211);
\draw[edge] (211) to node[LabelStyleH]{$[2]$} (221);
\draw[edge] (221) to node[LabelStyleH]{$[3]$} (222);
\draw[edge] (121) to node[LabelStyleH]{$ $} (122);
\draw[edge] (121) to node[LabelStyleH]{$ $} (221);
\draw[edge] (112) to node[LabelStyleH]{$ $} (212);
\draw[edge] (212) to node[LabelStyleH]{$[2]$} (222);
\draw[edge] (122) to node[LabelStyleH]{$ $} (222);

\foreach \y in {-.2}{
  \node[vertex] (t111) at (2.8,0+1.45*\y) {\small $\left(\tableau{ 1 & 1 & 1}, \tableau{1 & 2 & 3}\right)$};

    \node[vertex] (t112) at (2,-1 +1.65*\y) {\small $\left(\tableau{ 1 & 1 \cr 2}, \tableau{1 & 3 \cr 2}\right)$};
    \node[vertex] (t121) at (2, 0+1.65*\y) {\small $\left(\tableau{ 1 & 1 \cr 2}, \tableau{1 & 2 \cr 3}\right)$};
    \node[vertex] (t211) at (2, 1+1.45*\y) {\small $\left(\tableau{ 1 & 1 & 2}, \tableau{1 & 2 & 3}\right)$};

    \node[vertex] (t122) at (1, -1 + 1.65*\y) {\small $\left(\tableau{ 1 & 2\cr 2}, \tableau{1 & 2 \cr 3}\right)$};
    \node[vertex] (t212) at (1 , 0 +1.65*\y) {\small $\left(\tableau{ 1 & 2\cr 2}, \tableau{1 & 3\cr 2}\right)$};
    \node[vertex] (t221) at (1, 1+1.45*\y) {\small $\left(\tableau{ 1 & 2 & 2}, \tableau{1 & 2 & 3}\right)$};

    \node[vertex] (t222) at (0, 0+1.45*\y) {\small $\left(\tableau{2 & 2 & 2}, \tableau{1 & 2 & 3}\right)$};
}
\end{tikzpicture}
\caption{An illustration of Corollary \ref{c Schur-Weyl duality lower} for  $r=3, n=2$.  The pairs of tableaux are of the form $(P(\mathbf{k}^\dagger), Q(\mathbf{k}^\dagger))$.  The arrows and their coefficients give the action of $F_1$ on the lower canonical basis.}
\label{f c_111 example lower}
\end{figure}
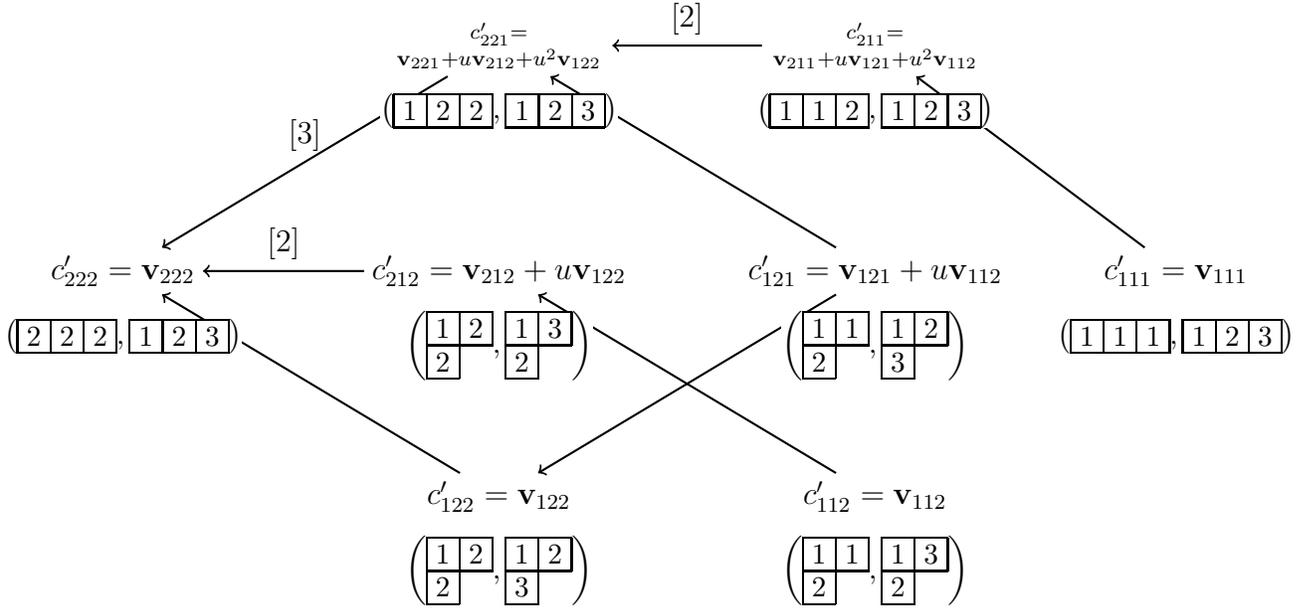

\begin{figure}[H]
\begin{tikzpicture}[xscale = 5,yscale = 3]
\tikzstyle{vertex}=[inner sep=0pt, outer sep=3pt, fill = white]
\tikzstyle{edge} = [draw, thick, ->,black]
\tikzstyle{LabelStyleH} = [text=black, anchor=south]
\tikzstyle{LabelStyleV} = [text=black, anchor=east]

    \node[vertex] (111) at (3,0) {$c_{111} = \bv_{111} $};

    \node[vertex] (112) at (2,-1) {$c_{112} = \bv_{112}$};
    \node[vertex] (121) at (2, 0) {$c_{121} = \bv_{121} - \ui \bv_{112} $};
    \node[vertex] (211) at (2, 1) {$c_{211}= \bv_{211}- \ui \bv_{121}$};

    \node[vertex] (122) at (1, -1) {$c_{122}= \bv_{122} $};
    \node[vertex] (212) at (1 , 0) {$c_{212}  = \bv_{212}- \ui \bv_{122} $};
    \node[vertex] (221) at (1, 1) {$c_{221}= \bv_{221}- \ui \bv_{212}$};

    \node[vertex] (222) at (0, 0) {$c_{222}= \bv_{222} $};

\draw[edge] (111) to node[LabelStyleH]{$[3]$} (112);
\draw[edge] (111) to node[LabelStyleH]{$[2]$} (121);
\draw[edge] (111) to node[LabelStyleH]{$ $} (211);
\draw[edge] (112) to node[LabelStyleH]{$[2]$} (122);
\draw[edge] (112) to node[LabelStyleH]{$ $} (212);
\draw[edge] (121) to node[LabelStyleH]{$ $} (221);
\draw[edge] (211) to node[LabelStyleH]{$ $} (212);
\draw[edge] (122) to node[LabelStyleH]{$ $} (222);

\foreach \y in {-.2}{
  \node[vertex] (t111) at (3,0+\y) {\small $\left(\tableau{ 1 & 1 & 1}, \tableau{1 & 2 & 3}\right)$};

    \node[vertex] (t112) at (2,-1+\y) {\small $\left(\tableau{ 1 & 1 & 2}, \tableau{1 & 2 & 3}\right)$};
    \node[vertex] (t121) at (2, 0+1.65*\y) {\small $\left(\tableau{ 1 & 1 \cr 2}, \tableau{1 & 2 \cr 3}\right)$};
    \node[vertex] (t211) at (2, 1+1.65*\y) {\small $\left(\tableau{ 1 & 1 \cr 2}, \tableau{1 & 3\cr 2}\right)$};

    \node[vertex] (t122) at (1, -1+\y) {\small $\left(\tableau{ 1 & 2 & 2}, \tableau{1 & 2 & 3}\right)$};
    \node[vertex] (t212) at (1 , 0 +1.65*\y) {\small $\left(\tableau{ 1 & 2\cr 2}, \tableau{1 & 3\cr 2}\right)$};
    \node[vertex] (t221) at (1, 1 +1.65*\y) {\small $\left(\tableau{ 1 & 2\cr 2}, \tableau{1 & 2 \cr 3}\right)$};

    \node[vertex] (t222) at (0, 0+\y) {\small $\left(\tableau{2 & 2 & 2}, \tableau{1 & 2 & 3}\right)$};
}
\end{tikzpicture}
\caption{An illustration of Corollary \ref{c Schur-Weyl duality upper} for  $r=3, n=2$.  The pairs of tableaux are of the form $(P(\mathbf{k}), Q(\mathbf{k}))$.  The arrows and their coefficients give the action of $F_1$ on the upper canonical basis.}
\label{f c_111 example upper}
\end{figure}
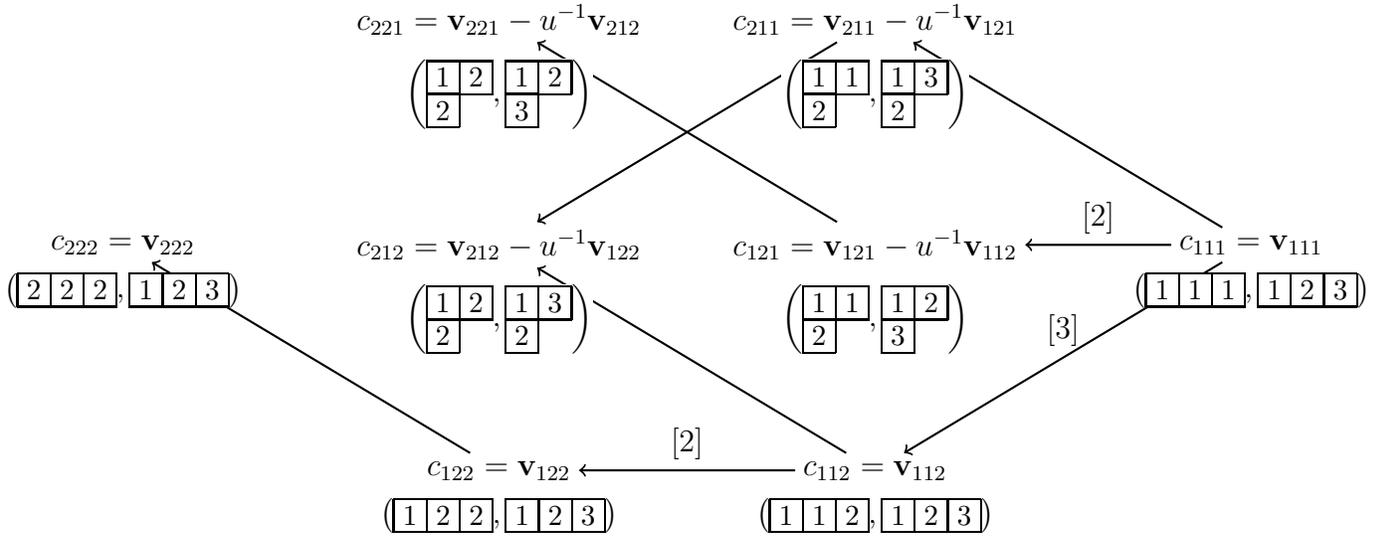

We assemble some equivalent descriptions of the lower canonical basis of $\bT$, which are also shown in \cite{Brundan} and appear in a slightly different form in \cite{GL, FKK}.
\begin{theorem}
\label{t Schur-Weyl duality lower}
The lower canonical basis element $\cp_\mathbf{k}, \, \mathbf{k} \in [n]^r,$ has the following equivalent descriptions
\begin{list}{\emph{(\roman{ctr})}} {\usecounter{ctr} \setlength{\itemsep}{1pt} \setlength{\topsep}{2pt}}
\item the unique $\br{\cdot}$-invariant element of $\bT_{\ZZ[\u]}$ congruent to $\bv_\mathbf{k} \mod \u \bT_{\ZZ[\u]}$;
\item $v_{k_1} \diamond \dots \diamond v_{k_r}$;
\item $G'(b'_\mathbf{k})$, where $G'$ is the inverse of the canonical isomorphism
\[(\QQ \tsr_\ZZ \bT_\mathbf{A}) \cap \L_0\cap\br{\L_0} \xrightarrow{\cong} \L_0 / \u \L_0; \]
\item The image of $\C_{D(\mathbf{k})}$ under the isomorphism in Proposition \ref{p weight space equals induced} ($D(\mathbf{k})$ is a maximal coset representative, defined in  \textsection\ref{ss type A combinatorics preliminaries}).
\end{list}
\end{theorem}
\begin{proof}
Description (i) is the definition of (ii) (Theorem \ref{t tsr product lower canbas}) and the element in (iii) is easily seen to satisfy the conditions in (i) (see Remark \ref{r balanced triple} and Theorem \ref{t theorem 6.2.2 HK}). The element in (iv) satisfies the conditions in (i) by Proposition \ref{p cell isomorphism induced} and the combination of Proposition \ref{p weight space equals induced} and \eqref{e br involution on H intertwines}. Note that we are actually applying an easy modification of Proposition \ref{p cell isomorphism induced} with  $\u$ in place of  $\ui$ and  $\C_v \tsr_{\H(W_J)} \br{T}_x$ in place of $\C_v \tsr_{\H(W_J)} T_x $.
\end{proof}

Proposition \ref{p cell isomorphism induced} and the discussion in  \textsection\ref{ss cell label conventions C_Q C'_Q}, results of \cite[Chapter 27]{LBook} (see the discussion at the end of  \textsection\ref{ss global canonical bases}), and the well-known combinatorics of the crystal basis $\B'$ (see e.g. \cite[Chapter 7]{HK}) allow us to determine the $\H_r$- and $\Uq$-cells of $(\bT, B')$.
\begin{corollary}
\label{c Schur-Weyl duality lower}
\
\begin{list}{\emph{(\roman{ctr})}} {\usecounter{ctr} \setlength{\itemsep}{1pt} \setlength{\topsep}{2pt}}
\item The $ \H_r$-module with basis $(\bT, B')$ decomposes into $ \H_r$-cells as $B' = \bigsqcup_{T \in \text{SSYT}_{[n]}^r} \Gamma'_T$, where
$\Gamma'_T = \{\cp_\mathbf{k} : P(\mathbf{k}^\dagger) = T \}.$
\item The $ \H_r$-cell $ \Gamma'_T$ of  $\bT$ is isomorphic to $\Gamma'_{\sh(T)}$.
\item The  $\Uq$-module with basis $(\bT, B')$ decomposes into $\Uq$-cells as $B' = \bigsqcup_{T \in \text{SYT}^r_{\leq n}} \Lambda'_T, $ where
$\Lambda'_T  = \{\cp_\mathbf{k} : Q(\mathbf{k}^\dagger) = T \}.$
\item The $\Uq$-cell $ \Lambda'_T$ is isomorphic to $B'(\sh(T))$.
\end{list}
\end{corollary}

\subsection{Upper canonical basis of $\bT$}
By Theorem \ref{t tsr product upper canbas} and associativity of the $\heart$ product, $(\bT, B)$ is an upper based module with balanced triple $(\QQ \tsr_\ZZ \bT_\mathbf{A}, \br{\L_\infty}, \L_\infty)$, where
\be
\begin{array}{ccl}
\L_\infty &:=& \L_\infty(\epsilon_1) \tsr_{\field_\infty} \dots \tsr_{\field_\infty} \L_\infty(\epsilon_1), \\
\B &:=& \B(\epsilon_1) \times \dots \times \B(\epsilon_1) \subseteq \L_\infty / \ui \L_\infty, \\
\bT_{\ZZ[\ui]} &:=& \ZZ[\ui] \{\bv_\mathbf{k} : \mathbf{k} \in [n]^r\}, \\
B &:=& B(\epsilon_1) \heart \dots \heart B(\epsilon_1).
\end{array}
\ee
We call $B$ the \emph{upper canonical basis} of $\bT$ and, for each  $\mathbf{k} \in [n]^r$, we write $c_\mathbf{k}$ for the element $v_{k_1} \heart \dots \heart v_{k_r} \in B$ and $b_\mathbf{k} \in \B$ for its image in $\L_\infty / \ui \L_\infty$.
 %$(\L_0, \B')$ is a lower crystal basis of $\bT$ at $\u = 0$.
Figure \ref{f c_111 example upper} gives the upper canonical basis in terms of the monomial basis for $r = 3, n =2$.

% In \cite{Brundan}, the element $c_\mathbf{k}$ is denoted $L_\alpha$.
We assemble some equivalent descriptions of the upper canonical basis of $\bT$.  The proof is similar to the corresponding Theorem \ref{t Schur-Weyl duality lower} for the lower canonical basis.
\begin{theorem}
\label{t Schur-Weyl duality upper}
The upper canonical basis element $c_\mathbf{k}, \, \mathbf{k} \in [n]^r,$ has the following equivalent descriptions
\begin{list}{\emph{(\roman{ctr})}} {\usecounter{ctr} \setlength{\itemsep}{1pt} \setlength{\topsep}{2pt}}
\item the unique $\br{\cdot}$-invariant element of $\bT_{\ZZ[\ui]}$, congruent to $\bv_\mathbf{k} \mod \ui\bT_{\ZZ[\ui]}$;
\item $v_{k_1} \heart \dots \heart v_{k_r}$;
\item $G(b_\mathbf{k})$, where $G$ is the inverse of the canonical isomorphism
\[(\QQ \tsr_\ZZ \bT_\mathbf{A}) \cap \br{\L_\infty} \cap \L_\infty \xrightarrow{\cong} \L_\infty / \ui \L_\infty; \]
\item The image of $C_{d(\mathbf{k})}$ under the isomorphism in Proposition \ref{p weight space equals induced}.
\end{list}
\end{theorem}

We also have the upper canonical basis version of Corollary \ref{c Schur-Weyl duality lower}.
\begin{corollary}
\label{c Schur-Weyl duality upper}
\
\begin{list}{\emph{(\roman{ctr})}} {\usecounter{ctr} \setlength{\itemsep}{1pt} \setlength{\topsep}{2pt}}
\item The $ \H_r$-module with basis $(\bT, B)$ decomposes into $ \H_r$-cells as $B = \bigsqcup_{T \in \text{SSYT}_{[n]}^r} \Gamma_T$, where
$\Gamma_T = \{\cp_\mathbf{k} : P(\mathbf{k}) = T \}.$
\item The $ \H_r$-cell $ \Gamma_T$ of  $\bT$ is isomorphic to $ \Gamma_{\sh(T)}$.
\item The  $\Uq$-module with basis $(\bT, B)$ decomposes into $\Uq$-cells as $B = \bigsqcup_{T \in \text{SYT}^r_{\leq n}} \Lambda_T, $ where
$\Lambda_T  = \{\cp_\mathbf{k} : Q(\mathbf{k}) = T \}.$
\item The $\Uq$-cell $ \Lambda_T$ is isomorphic to $ B(\sh(T))$.
\end{list}
\end{corollary}

\subsection{A symmetric bilinear form on $\bT$}
There is a bilinear form  $(\cdot, \cdot)$ on $\bT$ under which the upper and lower canonical basis are dual and satisfies several other nice properties.
%Recall that  $\varphi$ is the antiautomorphism of  $\Uq$ given by $\varphi(E_i) = F_i$,  $\varphi(F_i) = E_i$,  $\varphi(K_i) = K_i$,
Let  $\dagger$ (resp. $\dagger^\text{op}$) be the automorphism (resp. antiautomorphism) of $\H_r$ determined by $T_i^{\dagger} = T_{n-i}$ (resp. $T_i^{\dagger^\text{op}} = T_{n-i}$).
\begin{propdef}\cite{Brundan}
\label{p dual bases}
There is a unique symmetric bilinear form $(\cdot, \cdot)$ on $\bT$ satisfying
\begin{list}{\emph{(\roman{ctr})}} {\usecounter{ctr} \setlength{\itemsep}{1pt} \setlength{\topsep}{2pt}}
\item  $(x \bv,\bv')= (\bv,\varphi(x) \bv')$ for any  $x \in \Uq$,  $\bv, \bv' \in \bT$,
\item  $(\mathbf{v} h,\bv')= (\bv,\bv' h^{\dagger^\text{op}})$ for any  $h \in \H_r$,  $\bv, \bv' \in \bT$,
\item  $(\bv_\mathbf{k}, \br{\bv}_{\mathbf{l}^\dagger}) = \delta_{\mathbf{k}, \mathbf{l}}$,
\item  $(c_\mathbf{k}, \cp_{\mathbf{l}^\dagger}) = \delta_{\mathbf{k}, \mathbf{l}}$.
\end{list}
\end{propdef}

\section{Projected canonical bases}
\label{s projected canonical bases}

Here we give several equivalent definitions of the projected counterparts of the lower and upper canonical basis of $\bT$.  This will be used in the next section to help us understand the transition matrices discussed in the introduction.
Note that by quantum Schur-Weyl duality (Theorem \ref{c Schur-Weyl duality basic}), $\varsigma_\lambda^\bT = s_\lambda^\bT$, $\iota_\lambda^\bT = i_\lambda^\bT$, and $\pi_\lambda^\bT = p_\lambda^\bT$.

\subsection{Projected upper canonical basis}
\label{ss projected upper canonical basis}

For some of our descriptions of the projected upper canonical basis, we need an integral form that is different from $\bT_\mathbf{A}$.
First note that by Corollary \ref{c Schur-Weyl duality upper} and \cite[\textsection 5.2]{Kas2},
\be \label{e bT lambda fact}
\bT[ \ld \lambda] = \field \{c_\mathbf{k}: \sh(\mathbf{k}) \ld \lambda\} \text{ and } \bT[ \ldneq \lambda] = \field \{c_\mathbf{k}: \sh(\mathbf{k}) \ldneq \lambda\}.
\ee
Further, applying $\varsigma_\lambda^\bT$ to the upper based module  $(\bT[ \ld \lambda],\{c_\mathbf{k}: \sh(\mathbf{k}) \ld \lambda\})$ yields the upper based module $(\bT[\lambda],\{\varsigma_\lambda^\bT(c_\mathbf{k}): \sh(\mathbf{k}) = \lambda\})$ with balanced triple
\be \label{e lambda triple}
(\QQ \tsr_\ZZ \bT_\mathbf{A}[\ld \lambda]/ \bT_\mathbf{A}[\ldneq \lambda], \br{\L_\infty[\ld \lambda]}/\br{\L_\infty[\ldneq \lambda]}, \L_\infty[\ld \lambda]/ \L_\infty[\ldneq \lambda]),
%= (\varsigma_\lambda^\bT(\bT_\mathbf{A}[\ld \lambda]), \varsigma_\lambda^\bT(\br{\L_\infty[\ld \lambda]}), \varsigma_\lambda^\bT(\L_\infty[\ld \lambda])).
\ee
where $\bT_\mathbf{A}[\ld \lambda]$ and $\L_\infty[\ld \lambda]$ (resp. $\bT_\mathbf{A}[\ldneq \lambda]$ and $\L_\infty[\ldneq \lambda]$) are the $\mathbf{A}$- and $\field_\infty$- span of  $\{c_\mathbf{k} : \sh(\mathbf{k}) \ld \lambda\}$ (resp. $\{c_\mathbf{k} : \sh(\mathbf{k}) \ldneq \lambda\}$).
Finally, define
\be \label{e T lambda definition upper}
\begin{array}{ccl}
(\bT_\mathbf{A})_\lambda & := & \pi_\lambda^\bT(\bT_\mathbf{A}[\ld \lambda]), \\
\L_{\infty \lambda} & := & \pi_\lambda^\bT(\L_\infty[\ld \lambda]), \\
\tilde{\bT}_\mathbf{A} & := & \bigoplus_\lambda (\bT_\mathbf{A})_\lambda.
\end{array}
\ee
%note that we cannot just take \tilde{\bT}_\mathbf{A} & := & \pi_\lambda (\bT_\mathbf{A})

Our next theorem is similar to results in \cite{GL} and \cite[\textsection 7]{Brundan}; those of \cite{GL} are proved using geometric methods, and those of \cite[\textsection 7]{Brundan} using results of \cite{Kas2, LBook} as is done here.
\begin{theorem}
\label{t lifted upper canonical basis}
Maintain the notation above and let $\mathbf{l} \in [n]^r$ and $\lambda = \sh(\mathbf{l})$. Set $\mathbf{j} = \text{RSK}^{-1}(Z_\lambda,Q(\mathbf{l}))$, where $Z_\lambda$ is the superstandard tableau of shape $\lambda$ (see  \textsection\ref{ss type A combinatorics preliminaries}).
Let $V_{Q(\mathbf{l})} = \Uq c_\mathbf{j}$ and $V^{\QQ \text{ up}}_{Q(\mathbf{l})}$ be the  $\QQA$-form of $V_{Q(\mathbf{l})}$ as in  \textsection\ref{ss global canonical bases}.
Then the triples in (b) and (c) are balanced and the projected upper canonical basis element $\liftc_{\mathbf{l}}$ has the following descriptions
\begin{list}{\emph{(\alph{ctr})}} {\usecounter{ctr} \setlength{\itemsep}{1pt} \setlength{\topsep}{2pt}}
\item the unique  $\br{\cdot}$-invariant element of $\tilde{\bT}_\mathbf{A}$ congruent to $\bv_\mathbf{l} \mod \ui \L_\infty$,
\item $\tilde{G}(b_{\mathbf{l}})$, where $\tilde{G}$ is the inverse of the canonical isomorphism
\[(\QQ \tsr_\ZZ\tilde{\bT}_\mathbf{A}) \cap \br{\L_{\infty}}\cap\L_{\infty}  \xrightarrow{\cong} \L_{\infty} / \ui \L_{\infty}, \]
%??? \pi_\lambda should get some updating here
\item $\tilde{G}_\lambda(\pi_\lambda(b_{\mathbf{l}}))$, where $\tilde{G}_\lambda$ is the inverse of the canonical isomorphism
\[(\QQ \tsr_\ZZ (\bT_\mathbf{A})_\lambda) \cap \br{\L_{\infty\lambda}}\cap \L_{\infty\lambda}  \xrightarrow{\cong} \L_{\infty\lambda} / \ui \L_{\infty\lambda}, \]
\item the global crystal basis element $G_\lambda(b_{P(\mathbf{l})})$ of $V_{Q(\mathbf{l})}$,
\item $\pi_\lambda^\bT (c_{\mathbf{l}})$,
\item $p_\lambda^\bT (c_\mathbf{l})$.
\end{list}
Then $\tilde{B} := \{\liftc_\mathbf{k}: \mathbf{k}\in [n]^r\}$ is the \emph{projected upper canonical basis of $\bT$} and $(\bT, \tilde{B})$ is an upper based module. Its $\Uq$- and  $\H_r$-cells are given by Corollary \ref{c Schur-Weyl duality upper} with $\liftc$  in place of $c$.
\end{theorem}
\begin{proof}
The triple in (c) is just the injective image of the triple in \eqref{e lambda triple} under $\iota_\lambda^\bT$, hence the triple in (c) is balanced and the elements in (c) and (e) are the same.  As noted earlier, $\pi_\lambda^\bT = p_\lambda^\bT$, so the elements in (e) and (f) are the same.

By \cite[\textsection 5.2]{Kas2} (see the discussion at the end of \textsection\ref{ss global canonical bases}) and Corollary \ref{c Schur-Weyl duality upper}, there is an isomorphism of upper based modules
\[
(\bT[\lambda],\{\varsigma_\lambda^\bT(c_\mathbf{k}): \sh(\mathbf{k}) = \lambda\})\cong \bigoplus_{Q \in \text{SYT}(\lambda)}(V_\lambda^Q, B^Q(\lambda)), \quad \varsigma_\lambda^\bT(c_\mathbf{k}) \mapsto G_\lambda^{Q(\mathbf{k})}(b^{Q(\mathbf{k})}_{P(\mathbf{k})}), 
\]
where  $V_\lambda^Q, B^Q(\lambda),$ etc. denote copies of  $V_\lambda,B(\lambda),$ etc. indexed by $ Q$.  It follows that the elements in (c) and (d) are the same.

To see that the triple in (b) is balanced and that the elements in (b) and (c) are the same, we must show that the triple in (b) is the direct sum
$ \bigoplus_{\mu \vdash_n r}(\QQ \tsr_\ZZ (\bT_\mathbf{A})_\mu , \br{\L_{\infty\mu}}, \L_{\infty\mu})$ of triples of the form in (c).
This amounts to showing the equality of upper crystal bases (we need that this is an equality, not just an isomorphism)
\be \label{e L infinity equality}
(\L_\infty, \B) = \bigoplus_{\mu \vdash_n r} (\L_{\infty\mu}, \{\pi_\mu(b_\mathbf{k}):\sh(\mathbf{k})=\mu\}).
\ee
This follows from the uniqueness of upper crystal bases and the fact that the restriction of both sides of \eqref{e L infinity equality} to $\{x \in \bT^\mu : E_i x = 0 \text{ for all } i \in [n-1]\}$ is
%what is in isomorphism of \field_\infty- modules? Are all such torsion free modules free? I believe this is answered by the Smith normal form.
$(\field_\infty \{c_{\mathbf{k}}: P(\mathbf{k}) = Z_\mu\}, \{b_{\mathbf{k}}: P(\mathbf{k}) = Z_\mu\})$.

Finally, we can show that the element in (a) is the same as the other descriptions. The minimal central idempotent  $p_\lambda$ is $\br{\cdot}$-invariant. This follows from the  $\br{\cdot}$-invariance of the upper canonical basis of $\H_r$ and the fact that an algebra involution must yield an involution of minimal central idempotents. Thus $p_\lambda^\bT (c_\mathbf{l})$ is $\br{\cdot}$-invariant and, by description (b), it satisfies the other requirements of (a). The uniqueness in (a) was not clear a priori but is now clear because the triple in (b) is balanced.

The statements about $(\bT, \tilde{B})$ are clear from (e), (f) and Corollary \ref{c Schur-Weyl duality upper}.
\end{proof}

\begin{figure}
\begin{tikzpicture}[xscale = 5,yscale = 3]
\tikzstyle{vertex}=[inner sep=0pt, outer sep=3pt, fill = white]
\tikzstyle{edge} = [draw, thick, ->,black]
\tikzstyle{LabelStyleH} = [text=black, anchor=south]
\tikzstyle{LabelStyleV} = [text=black, anchor=east]

    \node[vertex] (111) at (3,0) {$\liftc_{111} = c_{111} $};

    \node[vertex] (112) at (2,-1) {$\liftc_{112} = \atop c_{112}+ \frac{[2]}{[3]}c_{121} + \frac{1}{[3]}c_{211} $};
    \node[vertex] (121) at (2, 0) {$\liftc_{121} = c_{121}$};
    \node[vertex] (211) at (2, 1) {$\liftc_{211}= c_{211}$};

    \node[vertex] (122) at (1, -1) {$\liftc_{122}= \atop c_{122} + \frac{[2]}{[3]}c_{212} + \frac{1}{[3]}c_{221}$};
    \node[vertex] (212) at (1 , 0) {$\liftc_{212}  = c_{212}$};
    \node[vertex] (221) at (1, 1) {$\liftc_{221}= c_{221}$};

    \node[vertex] (222) at (0, 0) {$\liftc_{222}= c_{222} $};

\draw[edge] (111) to node[LabelStyleH]{$[3]$} (112);
\draw[edge] (112) to node[LabelStyleH]{$[2]$} (122);
\draw[edge] (121) to node[LabelStyleH]{$ $} (221);
\draw[edge] (211) to node[LabelStyleH]{$ $} (212);
\draw[edge] (122) to node[LabelStyleH]{$ $} (222);

\foreach \y in {-.2}{
  \node[vertex] (t111) at (3,0+\y) {\small $\left(\tableau{ 1 & 1 & 1}, \tableau{1 & 2 & 3}\right)$};

    \node[vertex] (t112) at (2,-1+1.3*\y) {\small $\left(\tableau{ 1 & 1 & 2}, \tableau{1 & 2 & 3}\right)$};
    \node[vertex] (t121) at (2, 0+1.65*\y) {\small $\left(\tableau{ 1 & 1 \cr 2}, \tableau{1 & 2 \cr 3}\right)$};
    \node[vertex] (t211) at (2, 1+1.65*\y) {\small $\left(\tableau{ 1 & 1 \cr 2}, \tableau{1 & 3\cr 2}\right)$};

    \node[vertex] (t122) at (1, -1+1.3*\y) {\small $\left(\tableau{ 1 & 2 & 2}, \tableau{1 & 2 & 3}\right)$};
    \node[vertex] (t212) at (1 , 0 +1.65*\y) {\small $\left(\tableau{ 1 & 2\cr 2}, \tableau{1 & 3\cr 2}\right)$};
    \node[vertex] (t221) at (1, 1 +1.65*\y) {\small $\left(\tableau{ 1 & 2\cr 2}, \tableau{1 & 2 \cr 3}\right)$};

    \node[vertex] (t222) at (0, 0+\y) {\small $\left(\tableau{2 & 2 & 2}, \tableau{1 & 2 & 3}\right)$};
}
\end{tikzpicture}
\caption{The projected upper canonical basis elements of Theorem \ref{t lifted upper canonical basis} for  $r=3, n=2$.  The pairs of tableaux are of the form $(P(\mathbf{k}), Q(\mathbf{k}))$.  The arrows and their coefficients give the action of $F_1$ on the projected upper canonical basis.}
\label{f liftc_111 example upper}
\end{figure}
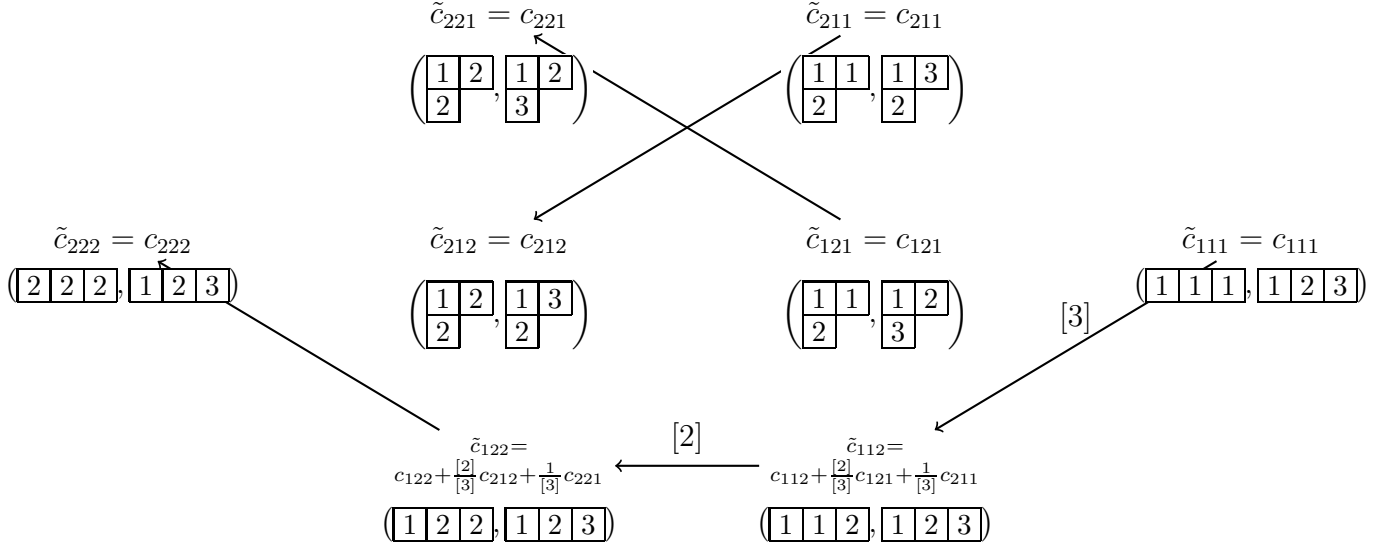

\begin{figure}
\begin{tikzpicture}[xscale = 5,yscale = 3]
\tikzstyle{vertex}=[inner sep=0pt, outer sep=3pt, fill = white]
\tikzstyle{edge} = [draw, thick, ->,black]
\tikzstyle{LabelStyleH} = [text=black, anchor=south]
\tikzstyle{LabelStyleV} = [text=black, anchor=east]
    \node[vertex] (111) at (2.8,0) {$\liftcp_{111} = \cp_{111} $};

    \node[vertex] (112) at (2,-1) {$\liftcp_{112} = \cp_{112} - \frac{1}{[3]}\cp_{211}$};
    \node[vertex] (121) at (2, 0) {$\liftcp_{121} = \cp_{121} - \frac{[2]}{[3]} \cp_{211} $};
    \node[vertex] (211) at (2, 1) {$\liftcp_{211}= \cp_{211}$};

    \node[vertex] (122) at (1, -1) {$\liftcp_{122}= \cp_{122} - \frac{1}{[3]}\cp_{221} $};
    \node[vertex] (212) at (1 , 0) {$\liftcp_{212}  = \cp_{212} -\frac{[2]}{[3]} \cp_{221} $};
    \node[vertex] (221) at (1, 1) {$\liftcp_{221}= \cp_{221}$};

    \node[vertex] (222) at (0, 0) {$\liftcp_{222}= \cp_{222} $};

\draw[edge] (111) to node[LabelStyleH]{$ $} (211);
\draw[edge] (211) to node[LabelStyleH]{$[2]$} (221);
\draw[edge] (221) to node[LabelStyleH]{$[3]$} (222);
\draw[edge] (121) to node[LabelStyleH]{$ $} (122);
\draw[edge] (112) to node[LabelStyleH]{$ $} (212);

\foreach \y in {-.2}{
  \node[vertex] (t111) at (2.8,0+1.45*\y) {\small $\left(\tableau{ 1 & 1 & 1}, \tableau{1 & 2 & 3}\right)$};

    \node[vertex] (t112) at (2,-1 +1.65*\y) {\small $\left(\tableau{ 1 & 1 \cr 2}, \tableau{1 & 3 \cr 2}\right)$};
    \node[vertex] (t121) at (2, 0+1.65*\y) {\small $\left(\tableau{ 1 & 1 \cr 2}, \tableau{1 & 2 \cr 3}\right)$};
    \node[vertex] (t211) at (2, 1+1.45*\y) {\small $\left(\tableau{ 1 & 1 & 2}, \tableau{1 & 2 & 3}\right)$};

    \node[vertex] (t122) at (1, -1 + 1.65*\y) {\small $\left(\tableau{ 1 & 2\cr 2}, \tableau{1 & 2 \cr 3}\right)$};
    \node[vertex] (t212) at (1 , 0 +1.65*\y) {\small $\left(\tableau{ 1 & 2\cr 2}, \tableau{1 & 3\cr 2}\right)$};
    \node[vertex] (t221) at (1, 1+1.45*\y) {\small $\left(\tableau{ 1 & 2 & 2}, \tableau{1 & 2 & 3}\right)$};

    \node[vertex] (t222) at (0, 0+1.45*\y) {\small $\left(\tableau{2 & 2 & 2}, \tableau{1 & 2 & 3}\right)$};
}
\end{tikzpicture}
\caption{The projected lower canonical basis elements of Theorem \ref{t lifted lower canonical basis} for  $r=3, n=2$.  The pairs of tableaux are of the form $(P(\mathbf{k}^\dagger), Q(\mathbf{k}^\dagger))$.  The arrows and their coefficients give the action of $F_1$ on the projected lower canonical basis.}
\label{f liftc 111 example lower}
\end{figure}
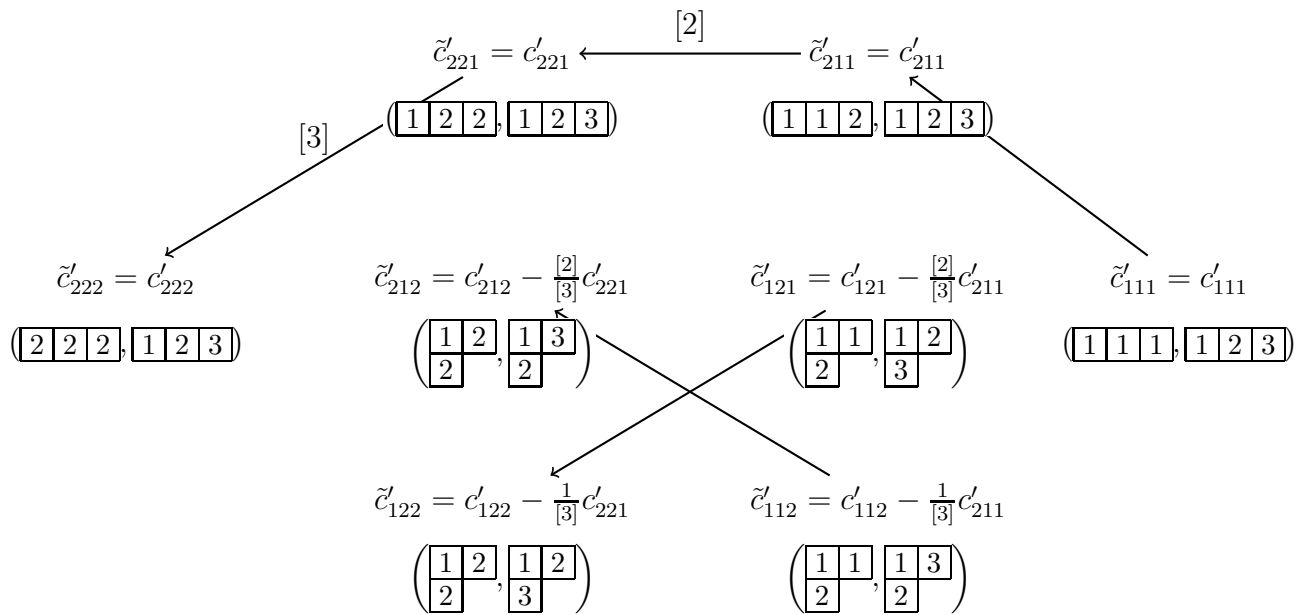

\subsection{Projected lower canonical basis}
Our equivalent descriptions of  the projected lower canonical basis are similar to those for the upper, with some minor changes.
By Corollary \ref{c Schur-Weyl duality lower} and \cite[Proposition 27.1.8]{LBook},
\be \label{e bT lambda fact lower}
\bT[ \gd \lambda] = \field \{\cp_\mathbf{k}: \sh(\mathbf{k}^\dagger) \gd \lambda\} \text{ and } \bT[ \gdneq \lambda] = \field \{\cp_\mathbf{k}: \sh(\mathbf{k}^\dagger) \gdneq \lambda\}.
\ee
%Further, applying $\varsigma_\lambda^\bT$ to the lower based module  $(\bT[ \gd \lambda],\{\cp_\mathbf{k}: \sh(\mathbf{k}^\dagger) \gd \lambda\})$ yields the lower based module $(\bT[\lambda],\{\varsigma_\lambda^\bT(\cp_\mathbf{k}): \sh(\mathbf{k}^\dagger) = \lambda\})$ with balanced triple
%\be \label{e lambda triple}
%(\bT_\mathbf{A}[\gd \lambda]/ \bT_\mathbf{A}[\gdneq \lambda], \br{\L_0[\gd \lambda]}/\br{\L_0[\gdneq \lambda]}, \L_0[\gd \lambda]/ \L_0[\gdneq \lambda]),
%%= (\varsigma_\lambda^\bT(\bT_\mathbf{A}[\gd \lambda]), \varsigma_\lambda^\bT(\br{\L_0[\gd \lambda]}), \varsigma_\lambda^\bT(\L_0[\gd \lambda])).
%\ee
Let $\bT_\mathbf{A}[\gd \lambda]$ and $\L_0[\gd \lambda]$ be the $\mathbf{A}$- and $\field_0$- span of  $\{\cp_\mathbf{k} : \sh(\mathbf{k}^\dagger) \gd \lambda\}$ and define
\be \label{e T lambda definition lower}
\begin{array}{ccl}
(\bT_\mathbf{A})'_\lambda & := & \pi_\lambda^\bT(\bT_\mathbf{A}[\gd \lambda]), \\
\L_{0 \lambda} & := & \pi_\lambda^\bT(\L_0[\gd \lambda]), \\
\tilde{\bT}'_\mathbf{A} & := & \bigoplus_\lambda (\bT_\mathbf{A})'_\lambda.
\end{array}
\ee
%note that we cannot just take \tilde{\bT}_\mathbf{A} & := & \pi_\lambda (\bT_\mathbf{A})

\begin{theorem}
\label{t lifted lower canonical basis}
Maintain the notation above and let $\mathbf{l} \in [n]^r$ and $\lambda = \sh(\mathbf{l}^\dagger)$. Set $\mathbf{j} = \text{RSK}^{-1}(Z_\lambda,Q(\mathbf{l}^\dagger))$.
Let $V_{Q(\mathbf{l}^\dagger)} = \Uq \pi_\lambda^\bT(\cp_\mathbf{j})$ and $V^{\QQ \text{ low}}_{Q(\mathbf{l}^\dagger)}$ be the  $\QQA$-form of $V_{Q(\mathbf{l}^\dagger)}$ as in  \textsection\ref{ss global canonical bases}.
Then the triples in (b) and (c) are balanced and the projected lower canonical basis element $\liftcp_{\mathbf{l}}$ has the following descriptions
\begin{list}{\emph{(\alph{ctr})}} {\usecounter{ctr} \setlength{\itemsep}{1pt} \setlength{\topsep}{2pt}}
\item the unique  $\br{\cdot}$-invariant element of $\tilde{\bT}'_\mathbf{A}$ congruent to $\bv_\mathbf{l} \mod \u \L_0$,
\item $\tilde{G}'(b_{\mathbf{l}})$, where $\tilde{G}'$ is the inverse of the canonical isomorphism
\[(\QQ \tsr_\ZZ \tilde{\bT}'_\mathbf{A}) \cap \L_{0} \cap \br{\L_{0}}  \xrightarrow{\cong} \L_{0} / \u \L_{0}, \]
\item $\tilde{G}'_\lambda(\pi_\lambda(b_{\mathbf{l}}))$, where $\tilde{G}'_\lambda$ is the inverse of the canonical isomorphism
\[(\QQ \tsr_\ZZ (\bT_\mathbf{A})'_\lambda) \cap \L_{0\lambda}\cap \br{\L_{0\lambda}}  \xrightarrow{\cong} \L_{0\lambda} / \u \L_{0\lambda}, \]
\item the global crystal basis element $G'_\lambda(b_{P(\mathbf{l}^\dagger)})$ of $V_{Q(\mathbf{l}^\dagger)}$,
\item $\pi_\lambda^\bT (\cp_{\mathbf{l}})$,
\item $p_\lambda^\bT (\cp_\mathbf{l})$.
\end{list}
Then $\tilde{B'} := \{\liftcp_\mathbf{k}: \mathbf{k}\in [n]^r\}$ is the \emph{projected lower canonical basis of $\bT$} and $(\bT, \tilde{B'})$ is a lower based module. Its $\Uq$- and  $\H_r$-cells are given by Corollary \ref{c Schur-Weyl duality lower} with $\liftcp$  in place of $\cp$.
\end{theorem}
%However the fact that
%\be \label{e pi lambda cp}
%\pi_\lambda^\bT(\cp_\mathbf{j}) = \cp_\mathbf{j} + \sum_{\sh(\mathbf{j'}^\dagger) \gdneq \lambda} \t'_{\mathbf{j}'\mathbf{j}} \cp_{\mathbf{j}'}, \ t'_{\mathbf{j}'\mathbf{j}} \in \field
%\ee
%(which follows from \eqref{e bT lambda fact lower})
%This together with the uniqueness of lower crystal lattices is enough to prove \eqref{e L 0 equality}:
%uniqueness means that $\L_0$ is determined by the intersections $\L_0 \cap \{x \in \bT^\mu : E_i x = 0 \text{ for all } i \in [n-1]\}$ for all $\mu \vdash_n r$; then \be \label{e lattices equal}
%\L_0 =  \bigoplus_{\stackrel{P(\mathbf{j'}^\dagger) = Z_\mu,}{\mu \vdash_n r}} \L_0(\mu) \sum_{\stackrel{P(\mathbf{j''}^\dagger) = Z_\mu,}{\mu = \sh(\mathbf{j'}^\dagger)}} a_{\mathbf{j''} \mathbf{j'}} \pi_\mu^\bT (\cp_{\mathbf{j'}}),
%\ee
%where $\L_0(\mu)x$  denotes the copy of $\L_0(\mu)$  whose intersection with $\bT^\mu$  is $\field_0 x$.   Since $\field_0$ is a principal ideal domain, we may assume that the matrix $(a_{\mathbf{j''} \mathbf{j'}})_{P(\mathbf{j''}^\dagger) = Z_\lambda = P(\mathbf{j'}^\dagger)}$ is triangular with powers of $\u$ on the diagonal.
\begin{proof}
The proof is similar to that of Theorem  \ref{t lifted upper canonical basis}, using results of \cite[Chapter 27]{LBook} in place of \cite[\textsection 5.2]{Kas2}.  Slightly more care is needed to prove that
\be \label{e L 0 equality}
\L_0 = \bigoplus_{\mu \vdash_n r} \L_{0\mu}
\ee
since $\pi_\lambda^\bT(c_\mathbf{j}) = c_\mathbf{j}$, whereas $\pi_\lambda^\bT(\cp_\mathbf{j}) \neq \cp_\mathbf{j}$ in general.
However,  uniqueness of lower crystal lattices is still enough:
uniqueness means that both sides of \eqref{e L 0 equality} are determined by their intersection with $\bT^\mu_\text{hw} := \{x \in \bT^\mu : E_i x = 0 \text{ for all } i \in [n-1]\}$ for all $\mu \vdash_n r$.  Since $\varsigma_\mu^\bT$ restricts to an isomorphism $\bT^\mu_\text{hw} \xrightarrow{\cong} \bT[\mu]^\mu$ and
\be
\varsigma_\mu^\bT(\L_0 \cap \bT^\mu_\text{hw}) = \field_0 \{\varsigma_\mu^\bT(\cp_\mathbf{k}): P(\mathbf{k}^\dagger) = Z_\mu \} = \varsigma_\mu^\bT(\L_{0\mu} \cap \bT^\mu_\text{hw}),
\ee
we have  $\L_0 \cap \bT^\mu_\text{hw}= \L_{0\mu} \cap \bT^\mu_\text{hw}$. The equality \eqref{e L 0 equality} follows.
\end{proof}

The most interesting part of  Theorems  \ref{t lifted upper canonical basis} and \ref{t lifted lower canonical basis} for us  is that, though the integral form  needed for the  upper  (resp. lower) canonical basis  and projected     upper  (resp. lower) canonical basis differ,  the  upper  (resp. lower) crystal lattices  are the same.   This has the following consequence.

\begin{corollary}
\label{c transition matrix projected to canbas}
The transition matrix from the projected upper (resp. lower) canonical basis to the upper (resp. lower) canonical basis is unitriangular and is the identity at  $\u = 0$ and  $\u = \infty$. Precisely,
\[
\begin{array}{ccl}
\liftc_\mathbf{k} &=& c_\mathbf{k} + \sum_{\sh(\mathbf{k}') \ldneq \sh(\mathbf{k})} t_{\mathbf{k}'\mathbf{k}} c_{\mathbf{k}'}, \\
\liftcp_\mathbf{k} &=& \cp_{\mathbf{k}} + \sum_{\sh(\mathbf{k}'^\dagger) \gdneq \sh(\mathbf{k}^\dagger)} t'_{\mathbf{k}'\mathbf{k}} \cp_{\mathbf{k}'},
\end{array}
\]
where the coefficients  $t_{\mathbf{k}'\mathbf{k}}, t'_{\mathbf{k}'\mathbf{k}}$ are $\br{\cdot}$-invariant and belong to $\u \field_0 \cap \ui \field_\infty$.
\end{corollary}
\begin{proof}
The constraints on dominance order follow from \eqref{e bT lambda fact} and \eqref{e bT lambda fact lower} using the expressions (e) of Theorems \ref{t lifted upper canonical basis} and \ref{t lifted lower canonical basis} for the projected canonical basis elements.  By the expression (a) of Theorem \ref{t lifted upper canonical basis},  $c_\mathbf{k} \equiv \liftc_\mathbf{k} \mod \ui \L_\infty$. Thus $t_{\mathbf{k}'\mathbf{k}} \in \ui \field_\infty$.  Since the upper canonical basis and projected upper canonical basis are $\br{\cdot}$-invariant,
% by Theorem \ref{t lifted upper canonical basis} (a)
so are the entries of the transition matrix between them.
This further implies $t_{\mathbf{k}'\mathbf{k}} \in \u \field_0 \cap \ui \field_\infty$.
The proof for the lower canonical basis is similar.
\end{proof}

\subsection{Projected canonical bases are dual under the bilinear form}
\begin{proposition}
\label{p dual lifted bases}
The projected upper and lower canonical basis of $\bT$ are dual under $(\cdot, \cdot)$: there holds $(\liftc_\mathbf{k}, \liftcp_{\mathbf{l}^\dagger}) = \delta_{\mathbf{k}, \mathbf{l}}$ for all $\mathbf{k}, \mathbf{l} \in [n]^r$.
\end{proposition}
\begin{proof}
Let $\mathbf{k},\mathbf{l} \in [n]^r$ and $\lambda = \sh(\mathbf{k}), \mu = \sh(\mathbf{l})$.  If $\lambda = \mu$, then by the unitriangularities established in Corollary \ref{c transition matrix projected to canbas} together with the fact that the upper canonical basis is dual to the lower canonical basis, we have $(\liftc_\mathbf{k}, \liftcp_{\mathbf{l}^\dagger}) = \delta_{\mathbf{k}, \mathbf{l}}$.  In the case  $\lambda \neq \mu$, we use Proposition-Definition \ref{p dual bases} (ii) to conclude
\be (\liftc_\mathbf{k}, \liftcp_{\mathbf{l}^\dagger}) = ( \liftc_\mathbf{k} p_\lambda, \liftcp_{\mathbf{l}^\dagger} p_\mu)
= ( \liftc_\mathbf{k}, \liftcp_{\mathbf{l}^\dagger} p_\mu p_\lambda^{\dagger^\text{op}})
= ( \liftc_\mathbf{k}, \liftcp_{\mathbf{l}^\dagger} p_\mu p_\lambda )
= 0.
\ee
\end{proof}

\section{Consequences for the canonical bases of $M_\lambda$}
\label{s Consequences for the canonical bases of M_lambda}
We use the results of the previous section to understand projected canonical bases of $\H_r$ and the relation between the upper and lower canonical basis of $ M_\lambda$.  We will come across several transition matrices whose entries lie in $\field_0$ and are the identity at  $\u = 0$.  Define, for an element $f \in \u\field_0$, the \emph{leading coefficient} of $f$, denoted $\mu(f)$, to be the coefficient of $\u$ in the power series expansion of $f$.  It turns out that the leading coefficients of  many of these transition matrix entries coincide with the $\S_r$-graph edge weights $\mu(v,w)$.
% note that $\field_0$ is smaller than the power series ring $\QQ[[\u]]$.

\subsection{Projected canonical bases of $\H_r$}
\label{ss projected canonical bases of the Hecke algebra}
Here we define projected canonical bases of $\H_r$, which are essentially a special case of the projected canonical bases in the previous section.  For each $w \in \S_r$, the projected upper (resp. lower) canonical basis element $\liftC_w \in \H_r$ (resp.  $\liftCp_w$) is defined to be $C_w p_\lambda$ (resp.  $\C_w p_\lambda$), where $\lambda = \sh(w)$ (resp.  $\lambda = \sh(w^\dagger)$).

\begin{corollary}
\label{c lift at u = 0 in H}
The transition matrix $\liftT = (\liftT_{w'w})_{w',w \in \S_r}$ (resp. $\liftT' = (\liftT'_{w'w})_{w',w \in \S_r}$) expressing the projected upper (resp. lower) canonical basis of $\H_r$ in terms of the upper (resp. lower) canonical basis of $\H_r$
\begin{list}{\emph{(\roman{ctr})}} {\usecounter{ctr} \setlength{\itemsep}{1pt} \setlength{\topsep}{2pt}}
\item is unitriangular: $\liftC_w = C_w + \sum_{\sh(w') \ldneq \sh(w)} \liftT_{w'w} C_{w'}$

(resp. $\liftCp_w = \C_w + \sum_{\sh(w^{\prime\dagger}) \gdneq \sh(w^\dagger)} \liftT'_{w'w} \C_{w'}$),
\item has entries that are $\br{\cdot}$-invariant and belong to $\field_0 \cap \field_\infty$,
\item is the identity at  $\u = 0$ and  $\u = \infty$,
\item satisfies: $\mu(\liftT_{w'w}) = \mu(w',w)$ (resp. $\mu(\liftT'_{w'w}) = -\mu(w',w)$) for $w',w$ such that $P(w') \neq P(w)$ and $R(w') \setminus R({w}) \neq \emptyset$.
\end{list}
\end{corollary}
%Also,  $b_{\mathbf{k}' \mathbf{k}} = (-1)^{\ell(\mathbf{k}') - \ell(\mathbf{k})} a_{ \mathbf{k}' \mathbf{k}}$.
%The last statement can be deduced using the antiautomorphism  $\theta : \H_r \to \H_r$ satisfying  $\theta(C_w) = (-1)^{\ell(w)} \C_w$ (see \cite{B4}).  We omit the details.
\begin{proof}
Choose $n = r$ and set $\varepsilon = \epsilon_1 + \dots + \epsilon_r$. Then by Proposition \ref{p weight space equals induced}, Theorem \ref{t Schur-Weyl duality lower} (iv), and Theorem \ref{t Schur-Weyl duality upper} (iv), $\H_r \cong \bT_\mathbf{A}^{\varepsilon}$ as right $\H_r$-modules and under this isomorphism canonical bases are sent to canonical bases and projected canonical bases are sent to projected canonical bases.  Thus (i)-(iii) are a special case of Corollary \ref{c transition matrix projected to canbas}.

To prove (iv), we compute $\liftC_w C_s$, for $w \in \S_r$ such that $s \notin R(w)$, in terms of the upper canonical basis in two different ways:
\be
\liftC_w C_s = (\sum_{w'} \liftT_{w'w} C_{w'}) C_s = -[2]\sum_{\{w' :s \in R({w'})\}} \liftT_{w'w}C_{w'} + \sum_{\stackrel{w',w''}{s \not\in R({w'}), s \in R({w''})}} \liftT_{w'w}\mu(w'', w')C_{w''}.
\ee
On the other hand, since  $\mathbf{A} \{\liftC_{w'}: P(w') = P(w)\}$ and the cellular subquotient $\mathbf{A}\Gamma_{P(w)}$ are isomorphic as modules with basis,
\be
\liftC_w C_s = \sum_{\big\{w' : {s \in R({w'}), \atop P(w') = P(w)}\big\}} \mu(w',w)\liftC_{w'} = \sum_{\big\{w' : {s \in R({w'}), \atop P(w') = P(w)}\big\}} \mu(w',w) \sum_{w''} \liftT_{w''w'} C_{w''}.
\ee
Then for any $w''$ such that $s \in R({w''})$ and $P(w'') \neq P(w)$, equating coefficients of $C_{w''}$ yields
\[
0 = \sum_{\big\{w' : {s \in R({w'}), \atop P(w') = P(w)}\big\}} \liftT_{w''w'} \mu(w',w) + [2]\liftT_{w''w} - \sum_{\{w' : s \not\in R({w'}) \}} \mu(w'',w')\liftT_{w'w} \equiv \mu(\liftT_{w''w}) - \mu(w'',w),
\]
where the equivalence is mod $\u\field_0$ and uses (iii). This proves (iv) for the upper canonical basis. The proof for the lower canonical basis is similar.
\end{proof}

\begin{figure}
{\tiny\[
\begin{array}{ccccccccccccccc}
    &\liftC_{1234} & \liftC_{1324} & \liftC_{2 1 3 4 } & \liftC_{1 2 4 3 } & \liftC_{1 4 2 3 } & \liftC_{1 3 4 2 } & \liftC_{2 3 1 4 } & \liftC_{3 1 2 4 } & \liftC_{2 1 4 3 } & \liftC_{2 4 1 3 } & \liftC_{4 1 2 3 } & \liftC_{2 3 4 1 } & \liftC_{3 1 4 2 } & \liftC_{3 4 1 2 } \\
C_{1234} &     1 & 0 & 0 & 0 & 0 & 0 & 0 & 0 & 0 & 0 & 0 & 0 & 0 & 0 \\
C_{1324} &    \frac{[2]^2}{[4]}& 1 & 0 & 0 & 0 & 0 & 0 & 0 & 0 & 0 & 0 & 0 & 0 & 0 \\
C_{2134} &    \frac{[3]}{[4]}& 0 & 1 & 0 & 0 & 0 & 0 & 0 & 0 & 0 & 0 & 0 & 0 & 0 \\
C_{1243} &    \frac{[3]}{[4]}& 0 & 0 & 1 & 0 & 0 & 0 & 0 & 0 & 0 & 0 & 0 & 0 & 0 \\
C_{1423} &    \frac{[2]}{[4]}& 0 & 0 & 0 & 1 & 0 & 0 & 0 & 0 & 0 & 0 & 0 & 0 & 0 \\
C_{1342} &    \frac{[2]}{[4]}& 0 & 0 & 0 & 0 & 1 & 0 & 0 & 0 & 0 & 0 & 0 & 0 & 0 \\
C_{2314} &    \frac{[2]}{[4]}& 0 & 0 & 0 & 0 & 0 & 1 & 0 & 0 & 0 & 0 & 0 & 0 & 0 \\
C_{3124} &    \frac{[2]}{[4]}& 0 & 0 & 0 & 0 & 0 & 0 & 1 & 0 & 0 & 0 & 0 & 0 & 0 \\
C_{2143} &    \frac{[2]^3}{[3][4]}& 0 &\frac{1}{[2]}&\frac{1}{[2]}& 0 & 0 & 0 & 0 & 1 & 0 &\frac{1}{[2]}&\frac{1}{[2]}& 0 & 0 \\
C_{2413} &    \frac{[2]^2}{[3][4]}& 0 & 0 & 0 &\frac{1}{[2]}& 0 &\frac{1}{[2]}& 0 & 0 & 1 & 0 & 0 & 0 & 0 \\
C_{4123} &    \frac{1}{[4]}& 0 & 0 & 0 & 0 & 0 & 0 & 0 & 0 & 0 & 1 & 0 & 0 & 0 \\
C_{2341} &    \frac{1}{[4]}& 0 & 0 & 0 & 0 & 0 & 0 & 0 & 0 & 0 & 0 & 1 & 0 & 0 \\
C_{3142} &    \frac{[2]^2}{[3][4]}& 0 & 0 & 0 & 0 &\frac{1}{[2]}& 0 &\frac{1}{[2]}& 0 & 0 & 0 & 0 & 1 & 0 \\
C_{3412} &    \frac{[2]}{[3][4]}&\frac{1}{[2]}& 0 & 0 & 0 & 0 & 0 & 0 & 0 & 0 & 0 & 0 & 0 & 1 \\
C_{1432} &    \frac{[2]^2}{[3][4]}&\frac{-1}{[2][4]}& 0 &-\frac{[2]}{[4]}&\frac{[3]}{[4]}&\frac{[3]}{[4]}& 0 & 0 & 0 & 0 & 0 & 0 & 0 &\frac{1}{[2]}\\
C_{3214} &    \frac{[2]^2}{[3][4]}&\frac{-1}{[2][4]}&-\frac{[2]}{[4]}& 0 & 0 & 0 &\frac{[3]}{[4]}&\frac{[3]}{[4]}& 0 & 0 & 0 & 0 & 0 &\frac{1}{[2]}\\
C_{2431} &    \frac{[2]}{[3][4]}& 0 & 0 &-\frac{1}{[4]}&\frac{[3]}{[2][4]}& 0 &\frac{-1}{[2][4]}& 0 & 0 &\frac{1}{[2]}& 0 &\frac{[3]}{[4]}& 0 & 0 \\
C_{4132} &    \frac{[2]}{[3][4]}& 0 & 0 &-\frac{1}{[4]}& 0 &\frac{[3]}{[2][4]}& 0 &\frac{-1}{[2][4]}& 0 & 0 &\frac{[3]}{[4]}& 0 &\frac{1}{[2]}& 0 \\
C_{4213} &    \frac{[2]}{[3][4]}& 0 &-\frac{1}{[4]}& 0 &\frac{-1}{[2][4]}& 0 &\frac{[3]}{[2][4]}& 0 & 0 &\frac{1}{[2]}&\frac{[3]}{[4]}& 0 & 0 & 0 \\
C_{3241} &    \frac{[2]}{[3][4]}& 0 &-\frac{1}{[4]}& 0 & 0 &\frac{-1}{[2][4]}& 0 &\frac{[3]}{[2][4]}& 0 & 0 & 0 &\frac{[3]}{[4]}&\frac{1}{[2]}& 0 \\
C_{4312} &    \frac{1}{[3][4]}&\frac{[3]}{[2][4]}& 0 & 0 &-\frac{1}{[4]}& 0 & 0 &-\frac{1}{[4]}& 0 & 0 &\frac{[2]}{[4]}& 0 & 0 &\frac{1}{[2]}\\
C_{3421} &    \frac{1}{[3][4]}&\frac{[3]}{[2][4]}& 0 & 0 & 0 &-\frac{1}{[4]}&-\frac{1}{[4]}& 0 & 0 & 0 & 0 &  \frac{[2]}{[4]}& 0 &\frac{1}{[2]}\\
C_{4231} &    \frac{1}{[3][4]}& 0 &\frac{-1}{[2][4]}&\frac{-1}{[2][4]}& 0 & 0 & 0 & 0 &\frac{1}{[2]}& 0 &\frac{[3]}{[2][4]}&\frac{[3]}{[2][4]}& 0 & 0 \\
C_{4321} &    \frac{1}{[2][3][4]}&\frac{1}{[4]}& 0 & 0 &\frac{-1}{[2][4]}&\frac{-1}{[2][4]}&\frac{-1}{[2][4]}&\frac{-1}{[2][4]}&\frac{1}{[3]}&\frac{-1}{[2][3]}&\frac{1}{[4]}&\frac{1}{[4]}&\frac{-1}{[2][3]}&  \frac{1}{[3]}\\
\end{array}\]}
\caption{Some of the projected upper canonical basis elements $\liftC_w$ in terms of the upper canonical basis $\{C_v\}_{v \in S_4}$.}
\label{f liftC example}
\end{figure}

\begin{remark}
It is sensible to ask whether Corollary \ref{c lift at u = 0 in H} and other results in this section hold for other finite Coxeter groups  $W$ in place of $\S_r$ (perhaps with a slight modification if right cells do not correspond to $\CC(\u) \tsr_\mathbf{A} \H(W)$-irreducibles).  We have not investigated this, but note that our proof will not extend easily as it depends on Schur-Weyl duality.
\end{remark}

\subsection{Seminormal bases}
\label{ss seminormal bases}
We wish to use the results about projected canonical bases to understand the transition matrix between the lower canonical basis of $M_\lambda$ and the upper canonical basis of $M_\lambda$. To do this, we relate both to seminormal bases in the sense of \cite{RamSeminormal}.  The transition matrices between the canonical bases of $M_\lambda$ and their corresponding seminormal bases also appear to be quite interesting---see the positivity conjectures in the next subsection.
\begin{definition}\label{d seminormal}
Given a chain of split semisimple $\field$-algebras $\field\cong H_1 \subseteq H_2 \subseteq \dots \subseteq H_r$ and an $H_r$-irreducible $N_\lambda$, a \emph{seminormal basis} of $N_\lambda$ is a $\field$-basis $B$ of $N_\lambda$ compatible with the restrictions in the following sense: there is a partition $B = B_{\mu^1} \sqcup \dots \sqcup B_{\mu^k}$ such that if $N_{\mu^i} = \field B_{\mu^i}$ then $N_\lambda = N_{\mu^1} \oplus \dots \oplus N_{\mu^k}$ as $H_{r-1}$-modules. Further, there is a partition of each $B_{\mu^i}$ that gives rise to a decomposition of $N_{\mu^i}$ into $H_{r-2}$-irreducibles, and so on, all the way down to $H_1$.
\end{definition}
Note that if the restriction of an $H_i$-irreducible to $H_{i-1}$ is multiplicity-free, then a seminormal basis is unique up to a diagonal transformation.

To construct seminormal bases corresponding to the upper and lower canonical basis of $M_\lambda$, first define, for any $J \subseteq S$, $(\liftC_Q)^J$ to be the projection of $C_Q$ onto the irreducible $\field \H_J$-module corresponding to the right cell of $\Res_{\field \H_J} \field\Gamma_{\lambda}$ containing $C_Q$. Define $(\liftCp_Q)^J$ similarly.  If $J = \{s_1, \dots, s_{r-2} \}$, then by \cite[\textsection 4]{B0},  $(\liftC_Q)^J$ (resp.  $(\liftCp_Q)^J$) is equal to $C_Q p_\mu$ (resp. $\C_Q p_\mu $), where $\mu = \sh(Q|_{[r-1]})$. Here, for a tableau  $Q$ and set $Z \subseteq \ZZ$, $Q|_Z$ denotes the subtableau of $Q$ obtained by removing the entries not in $Z$.

Define a total order $\ld$ on SYT$(\lambda)$ by declaring $Q' \ld Q$ if the numbers $k+1, \dots, r$ are in the same positions in $Q'$ and $Q$ and $\sh(Q'|_{[k-1]}) \gd \sh(Q|_{[k-1]})$; this $k$ is unique and we refer to it as $k(Q',Q)$.  This total order is the reverse of the last letter order defined in \cite{GM}.
% (we choose $\gdneq$ rather than $\ldneq$ for consistency with the partial order on right cells).
\begin{lemma} \label{l lift transition matrix}
For $J = \{ s_1, \dots, s_{r-2} \}$, the transition matrix expressing the projected basis $\{(\liftC_Q)^J : Q \in SYT(\lambda) \}$ in terms of the upper canonical basis of $M_\lambda$ is lower-unitriangular, is the identity at  $\u = 0$ and  $\u = \infty$, and has $\br{\cdot}$-invariant entries  (i.e. $(\liftC_Q)^J = C_Q + \sum_{Q' \gdneq Q} m_{Q' Q} C_{Q'}$, $m_{Q' Q} \in \u\field_0,\ \br{m_{Q' Q}} = m_{Q' Q}$). The transition matrix expressing the projected basis $\{(\liftCp_Q)^J : Q \in SYT(\lambda) \}$ in terms of the lower canonical basis of $M_\lambda$ satisfies the same properties except is upper-unitriangular instead of lower-unitriangular (i.e. $(\liftCp_Q)^J = \C_Q + \sum_{Q' \ldneq Q} m'_{Q' Q} \C_{Q'}$, $m'_{Q' Q} \in \u\field_0,\ \br{m'_{Q' Q}} = m'_{Q' Q}$).
\end{lemma}
\begin{proof}
This follows from Corollary \ref{c lift at u = 0 in H} and Proposition \ref{p restrict Wgraph}: identify $M_\lambda^\mathbf{A}$ and its upper canonical basis with $\mathbf{A} \Gamma_{Z_\lambda^*}$, where $Z_\lambda^*$ is the standard tableau with $1, 2, \dots, \lambda_1$ in the first row, $\lambda_1 + 1, \dots, \lambda_1 + \lambda_2$ in the second row, etc. Then $\Res_{\H_J} \mathbf{A}\Gamma_{Z_\lambda^*}$ has right cells labeled by the result of uninserting an outer corner of $Z_\lambda^*$ (see \cite[\textsection 4]{B0}). The key point is that all uninsertions of $Z_\lambda^*$ result in the entry $\lambda_1$ being kicked out, and therefore by Proposition \ref{p restrict Wgraph}, the $\H_J$-module with basis $\Res_{\H_J} \mathbf{A}\Gamma_{Z_\lambda^*} \subseteq \mathbf{A} \{ \C_{xv} : v \in (\S_r)_J \}$ is isomorphic to a cellular subquotient of $\Gamma_{\S_{r-1}}$ (here  $x = s_{\lambda_1} s_{\lambda_1 +1} \cdots s_{r-1}$). We can then apply Corollary \ref{c lift at u = 0 in H} with $r$ of the proposition set to $r-1$. Lower-unitriangularity follows from Corollary  \ref{t right cell partial order}.  The proof for the lower canonical basis is similar.
% why are the transition coefficients the sameafter we quotient by stuff outside the cell\Gamma_\lambda
% I checked this and I'm pretty sure it's right
\end{proof}

Set $J_i = \{s_1, \dots, s_{i-1}\}$.  We now define the upper seminormal basis to be $\{\gt_Q : Q \in \text{SYT}(\lambda)\}$, where $\gt_Q$ is the result of applying  the construction $C_Q \tto (\liftC_Q)^J$ first with $J = J_{r-1}$, then with $J =J_{r-2}$, and so on, finishing with $J = J_1 = \emptyset$.  The lower seminormal basis $\{\gtp_Q : Q \in \text{SYT}(\lambda)\}$ is defined similarly.  These bases are seminormal with respect to the chain $\H_{J_1} \subseteq \cdots \subseteq \H_{J_{r-1}} \subseteq \H_r$.
\begin{proposition} \label{p gt to canonical transition}
The transition matrix $T(\lambda) = (T_{Q' Q})_{Q', Q \in \text{SYT}(\lambda)}$ (resp. $T'(\lambda) = (T'_{Q' Q})_{Q', Q \in \text{SYT}(\lambda)}$) expressing the upper (resp. lower) seminormal basis of $ M_\lambda$ in terms of the upper (resp. lower) canonical basis of $M_\lambda$ and  $T(\lambda)^{-1}$ (resp.  $T'(\lambda)^{-1}$)
\begin{list}{\emph{(\roman{ctr})}} {\usecounter{ctr} \setlength{\itemsep}{1pt} \setlength{\topsep}{2pt}}
\item are lower-unitriangular: $\gt_Q = C_Q + \sum_{Q' \gdneq Q} T_{Q' Q} C_{Q'}$ and similarly for $T(\lambda)^{-1}$

\noindent(resp. upper-unitriangular: $\gtp_Q = \C_Q + \sum_{Q' \ldneq Q} T'_{Q' Q} \C_{Q'}$ and similarly for $T'(\lambda)^{-1}$),
\item have entries that are $\br{\cdot}$-invariant and belong to $\field_0 \cap \field_\infty$,
\item are the identity at  $\u = 0$ and  $\u = \infty$,
\item satisfy: $\mu(T_{Q'Q}) = \mu(Q',Q)$ and $\mu(T^{-1}_{Q'Q}) = -\mu(Q',Q)$ for $Q',Q$ such that $Q' \gdneq Q$ and $(R(C_{Q'}) \setminus R(C_{Q})) \cap  J_{k(Q',Q)-1} \neq \emptyset$

\noindent(resp. $\mu(T'_{Q'Q}) = -\mu(Q',Q)$ and $\mu(T'^{-1}_{Q'Q}) = \mu(Q',Q)$ for $Q',Q$ such that $Q' \ldneq Q$ and $(R(\C_{Q'}) \setminus R(\C_{Q})) \cap  J_{k(Q',Q)-1} \neq \emptyset$).
\end{list}
\end{proposition}
\begin{proof}
%We give the proof for the left-hand Gelfand-Tsetlin basis, but the same proof works for any Gelfand-Tsetlin basis.
%Given a sequence $i_1, i_2, \dots, i_l$ of distinct elements of $[r-1]$, let $L_j = S \backslash \{ s_{i_j}, s_{i_{j+1}}, \dots, s_{i_l} \}$, and require that the sequence is such that $\H_{L_j} \cong \H_t$ for some $t < r$.
The transition matrix $T(\lambda)$ is the product $\tilde{M}^{J_{r-1}} \tilde{M}^{J_{r-2}} \cdots \tilde{M}^{J_1}$, where $\tilde{M}^{J_i}$ is a block diagonal matrix, with each block of the form described in Lemma \ref{l lift transition matrix} with $J$ of the lemma equal to $J_i$. Properties (i)-(iii) of $T(\lambda)$ then follow because they are preserved under matrix multiplication and diagonally joining blocks.

To prove (iv), we apply the following easy claim
\refstepcounter{equation}
\begin{enumerate}[label={(\theequation)}]
\label{en matrix product mu}
\item if $M^1, \dots, M^l$ are matrices satisfying (iii) and $M = \prod_{k=1}^l M^k$, then $\mu(M_{ij}) = \sum_{k} \mu(M^k_{ij})$ for $i \neq j$.
\end{enumerate}
to obtain $\mu(T_{Q'Q}) = \sum_{k=1}^{r-1} \mu(\tilde{M}^{J_k}_{Q'Q})$.  If $Q' \gdneq Q$, then there is exactly one $k$ for which $\tilde{M}^{J_{k-1}}_{Q'Q}$ is non-zero; this $k$ is exactly  $k(Q',Q)$. Further, by Corollary \ref{c lift at u = 0 in H} (iv) and the proof of Lemma \ref{l lift transition matrix}, $\mu(\tilde{M}^{J_{k(Q',Q)-1}}_{Q'Q}) = \mu(Q',Q)$ if $(R(C_{Q'}) \setminus R(C_{Q})) \cap  J_{k(Q',Q)-1} \neq \emptyset$. This proves (iv) for $T(\lambda)$. The statements for $T'(\lambda)$ are proved similarly and the statements for $T(\lambda)^{-1}$ and $T'(\lambda)^{-1}$ follow easily.
\end{proof}

\begin{example}
Continuing Example \ref{ex lambda31}, we give transition matrices between the various bases defined above. The convention is that the columns of the matrix express the basis element at the top of the column in terms of the row labels. The matrices $D(\lambda)$ and $S(\lambda)$ are defined in Theorem \ref{t transition C' to C} and its proof (below).
\setlength{\cellsize}{8pt}
\[{\tiny
\begin{array}{cccc}
 \vspace{1pt}
 & \gt_{Q_4} & \gt_{Q_3} & \gt_{Q_2} \\
  \vspace{1pt}
C_{Q_4} & 1 & 0 & 0 \\
 \vspace{1pt}
C_{Q_3} & \frac{[2]}{[3]} & 1 & 0 \\
 \vspace{1pt}
C_{Q_2} & \frac{1}{[3]} & \frac{1}{[2]} & 1 \\
&\multicolumn{3}{c}{T((3,1))}
\end{array}} \quad
{\tiny
\begin{array}{cccc}
 \vspace{1pt}
 & \gtp_{Q_4} & \gtp_{Q_3} & \gtp_{Q_2} \\
  \vspace{1pt}
\gt_{Q_4} & [3] & 0 & 0 \\
 \vspace{1pt}
\gt_{Q_3} & 0 & \frac{[2][4]}{[3]} & 0 \\
 \vspace{1pt}
\gt_{Q_2} & 0 & 0 & \frac{[4]}{[2]}\\
& \multicolumn{3}{c}{D((3,1))}
\end{array}}\quad
{\tiny
\begin{array}{cccc}
 \vspace{1pt}
 & \C_{Q_4} & \C_{Q_3} & \C_{Q_2} \\
  \vspace{1pt}
\gtp_{Q_4} & 1 & \frac{[2]}{[3]} & \frac{1}{[3]} \\
 \vspace{1pt}
\gtp_{Q_3} & 0 & 1 & \frac{1}{[2]} \\
 \vspace{1pt}
\gtp_{Q_2} & 0 & 0 & 1 \\
& \multicolumn{3}{c}{T'((3,1))^{-1}}
\end{array}}
\]
\[{\tiny
\begin{array}{cccc}
 \vspace{2pt}
 & \C_{Q_4} & \C_{Q_3} & \C_{Q_2} \\
  \vspace{2pt}
C_{Q_4} & [3] & [2] & 1 \\
 \vspace{2pt}
C_{Q_3} & [2] & [2]^2 & [2] \\
 \vspace{2pt}
C_{Q_2} & 1 & [2] & [3]\\
& \multicolumn{3}{c}{S((3,1))}
\end{array}}
\]
The more substantial example $\lambda = (4,2)$ is below, where $S(\lambda)$ is scaled as in Conjecture \ref{cj non-negativity T T' D}, so that its entries lie in $\mathbf{A}$, are $\br{\cdot}$-invariant,  and have greatest common divisor 1.
\[{\tiny
\setlength{\cellsize}{8pt}
\begin{array}{cccccccccc}
 \vspace{2pt}
 & \tableau{1 & 2 & 3 & 4\\ 5 & 6} & \tableau{1 & 2 & 3 & 5\\ 4 & 6} & \tableau{1 & 2 & 4 & 5\\ 3 & 6} & \tableau{1 & 3 & 4 & 5\\ 2 & 6} & \tableau{1 & 2 & 3 & 6\\ 4 & 5} & \tableau{1 & 2 & 4 & 6\\ 3 & 5} & \tableau{1 & 3 & 4 & 6\\ 2 & 5} & \tableau{1 & 2 & 5 & 6\\ 3 & 4} & \tableau{1 & 3 & 5 & 6\\ 2 & 4} \\
 \vspace{2pt}
\tableau{1 & 2 & 3 & 4\\ 5 & 6} & 1 & \frac{[3]}{[4]} & \frac{[2]}{[4]} & \frac{[1]}{[4]} & \frac{[2]}{[4]} & \frac{[2]^2}{[4]}  & \frac{[2]}{[4]} &  \frac{[2]}{[4][3]} & \frac{[2]^2}{[3][4]} \\
 \vspace{2pt}
\tableau{1 & 2 & 3 & 5\\ 4 & 6} & 0 & 1 & \frac{[2]}{[3]} & \frac{1}{[3]} & \frac{[2]}{[3]} & \frac{[2]^2}{[3]^2} & \frac{[2]}{[3]^2} & \frac{[2]}{[3]^2} & \frac{[2]^2}{[3]^2} \\
 \vspace{2pt}
\tableau{1 & 2 & 4 & 5\\ 3 & 6} & 0 & 0 & 1 & \frac{1}{[2]} & 0 & \frac{[2]}{[3]} & \frac{1}{[3]} & \frac{1}{[3]} & \frac{1}{[2][3]} \\
 \vspace{2pt}
\tableau{1 & 3 & 4 & 5\\ 2 & 6} & 0 & 0 & 0 & 1 & 0 & 0 & \frac{[2]}{[3]} & 0 & \frac{1}{[3]} \\
 \vspace{2pt}
\tableau{1 & 2 & 3 & 6\\ 4 & 5} & 0&    0&    0&    0&    1&\frac{[2]}{[3]}& \frac{1}{[3]}& \frac{1}{[3]}&\frac{[2]}{[3]} \\
 \vspace{2pt}
\tableau{1 & 2 & 4 & 6\\ 3 & 5} & 0&    0&    0&    0&    0&    1& \frac{1}{[2]}& \frac{1}{[2]}  &  \frac{1}{[2]^2} \\
 \vspace{2pt}
\tableau{1 & 3 & 4 & 6\\ 2 & 5} & 0&    0&    0&    0&    0&    0&    1&    0& \frac{1}{[2]} \\
 \vspace{2pt}
\tableau{1 & 2 & 5 & 6\\ 3 & 4} & 0&    0&    0&    0&    0&    0&    0&    1& \frac{1}{[2]} \\
 \vspace{2pt}
\tableau{1 & 3 & 5 & 6\\ 2 & 4} & 0&    0&    0&    0&    0&    0&    0&    0&    1\\
& \multicolumn{9}{c}{T'((4,2))^{-1}}
\end{array}}\]
\vspace{3pt}
\[{\tiny
\begin{array}{cccccccccc}
 \vspace{2pt}
 & \tableau{1 & 2 & 3 & 4\\ 5 & 6} & \tableau{1 & 2 & 3 & 5\\ 4 & 6} & \tableau{1 & 2 & 4 & 5\\ 3 & 6} & \tableau{1 & 3 & 4 & 5\\ 2 & 6} & \tableau{1 & 2 & 3 & 6\\ 4 & 5} & \tableau{1 & 2 & 4 & 6\\ 3 & 5} & \tableau{1 & 3 & 4 & 6\\ 2 & 5} & \tableau{1 & 2 & 5 & 6\\ 3 & 4} & \tableau{1 & 3 & 5 & 6\\ 2 & 4} \\
  \vspace{2pt}
\tableau{1 & 2 & 3 & 4\\ 5 & 6} & [3][4] & [3]^2 & [2][3] & [3] & [2][3] & [2]^2[3] & [2][3] & [2] & [2]^2 \\
 \vspace{2pt}
\tableau{1 & 2 & 3 & 5\\ 4 & 6} & [3]^2 & [2][3]^2 & [2]^2[3] & [2][3] & [2]^2[3] & 2[4] + 3[2] & 2[3] + 1 & [2]^2 & [2]^3 \\
 \vspace{2pt}
\tableau{1 & 2 & 4 & 5\\ 3 & 6} & [2][3] & [2]^2[3] & [2][3]^2 & [3]^2 & [2]^3 & [2]^4 & [2]^3 & [2][3] & 2[3] + 1 \\
 \vspace{2pt}
\tableau{1 & 3 & 4 & 5\\ 2 & 6} & [3] & [2][3] & [3]^2 & [3][4] & [2]^2 & [2]^3 & [3]^2 & [3] & [2][3] \\
 \vspace{2pt}
\tableau{1 & 2 & 3 & 6\\ 4 & 5} & [2][3] & [2]^2[3] & [2]^3 & [2]^2 & [2][3]^2 & [2]^4 & [2]^3 & [2][3] & [2]^2[3] \\
 \vspace{2pt}
\tableau{1 & 2 & 4 & 6\\ 3 & 5} & [2]^2[3] & 2[4] + 3[2] & [2]^4 & [2]^3 & [2]^4 & [2]^5 & [2]^4 & [2]^2[3] & 2[4] + 3[2] \\
 \vspace{2pt}
\tableau{1 & 3 & 4 & 6\\ 2 & 5} & [2][3] & 2[3] + 1 & [2]^3 & [3]^2 & [2]^3 & [2]^4 & [2][3]^2 & [2][3] & [2]^2[3] \\
 \vspace{2pt}
\tableau{1 & 2 & 5 & 6\\ 3 & 4} & [2] & [2]^2 & [2][3] & [3] & [2][3] & [2]^2[3] & [2][3] & [3][4] & [3]^2 \\
 \vspace{2pt}
\tableau{1 & 3 & 5 & 6\\ 2 & 4} & [2]^2 & [2]^3 & 2[3] + 1 & [2][3] & [2]^2[3] & 2[4] + 3[2] & [2]^2[3] & [3]^2 & [2][3]^2\\
& \multicolumn{9}{c}{S((4,2))}
\end{array}}\]
\end{example}

For the next proposition, let us clarify a confusing point. Let $\Gamma'_T$, $T \in \text{SSYT}_{[n]}(\lambda)$, be the $\H_r$-cell of $\bT$ from Corollary \ref{c Schur-Weyl duality lower} (i); let $\Gamma'_P$ be the right cell of $\Gamma'_{\S_r}$ from \textsection \ref{ss cell label conventions C_Q C'_Q}, where $P \in \text{SYT}(\lambda)$ is the \emph{standardization} of $T$, i.e. $P = \transpose{P(D(\mathbf{k}))}$ for any $\mathbf{k}$ with $P(\mathbf{k}^\dagger) = T$. The $\H_r$-cells $\Gamma'_P, \Gamma'_T$, and $\Gamma'_\lambda$ give rise to isomorphic $\H_r$-modules with basis, the isomorphisms being given by
\be
\label{e sb involution}
\begin{array}{cclccclcccl}
\Gamma'_T &\cong& \Gamma'_P, && \Gamma'_P &\cong& \Gamma'_\lambda, && \Gamma'_T &\cong& \Gamma'_\lambda,\\
\cp_\mathbf{k} &\longleftrightarrow& \C_{D(\mathbf{k})} && \C_w &\longleftrightarrow& \C_{\transpose{Q(w)}} && \cp_\mathbf{k} &\longleftrightarrow& \C_{Q(\mathbf{k}^\dagger)^\dagger}
\end{array}
\ee
where $Q^\dagger$, for $Q$ a SYT, denotes the Sch\"utzenberger involution of $Q$ (see, e.g., \cite[A1.2]{F}). The left-hand isomorphism is from Theorem \ref{t Schur-Weyl duality lower} (iv), the middle from \textsection \ref{ss cell label conventions C_Q C'_Q}, and the right is the composition of the two.

Let $1^\text{op}$ be the antiautomorphism of  $\H_r$ determined by $T_i^{1^\text{op}} = T_i$.
\begin{proposition} \label{p dual bases Mlambda}
There is a bilinear form  $\langle \cdot,\cdot \rangle : M_\lambda \times M_\lambda \to \field$ satisfying
\begin{list}{\emph{(\roman{ctr})}} {\usecounter{ctr} \setlength{\itemsep}{1pt} \setlength{\topsep}{2pt}}
\item  $\langle x h,x' \rangle = \langle x,x' h^{1^\text{op}} \rangle $ for any  $h \in \H_r$,  $x, x' \in M_\lambda$,
\item $\langle  C_Q, \C_{Q'}  \rangle  = \delta_{QQ'}$,
\item $\langle  (\liftC_Q)^J, (\liftCp_{Q'})^J  \rangle  = \delta_{QQ'}$,
\item $\langle  \gt_Q, \gtp_{Q'}  \rangle  = \delta_{QQ'}$.
\end{list}
\end{proposition}
\begin{proof}
By Proposition \ref{p dual lifted bases}, the inner product on $\bT$ restricts to an inner product on $\pi^\bT_\lambda(\field \Gamma_T) \times \pi^\bT_\lambda(\field\Gamma'_T) \to \field$ for any  $T \in \text{SSYT}_{[n]}(\lambda)$.  This yields an inner product on $M_\lambda$ satisfying $( C_{Q}, \C_{Q'} ) = \delta_{QQ^{\prime \dagger}}$ (we have used the right-hand isomorphism of \eqref{e sb involution}).  Letting  $M_\lambda^\dagger$ denote the result of twisting $M_\lambda$ by the automorphism  $\dagger$, we have $M_\lambda \cong M_\lambda^\dagger$ via  $\C_Q \mapsto \C_{Q^\dagger}$. Applying this isomorphism to the second factor yields the inner product  $\langle \cdot, \cdot \rangle $ satisfying (i) and (ii).
Given (ii), the proof of (iii) is similar to that of Proposition \ref{p dual lifted bases}. Iterating this argument through the sequence of projections that yields the seminormal bases proves (iv).
\end{proof}
\begin{theorem}
\label{t transition C' to C}
The transition matrix $S(\lambda) = (S_{Q' Q})_{Q', Q \in \text{SYT}(\lambda)}$ expressing the lower canonical basis of $ M_\lambda$ in terms of the upper canonical basis of $M_\lambda$ (i.e. $\C_Q = \sum_{Q' \in \text{SYT}(\lambda)} S_{Q' Q} C_{Q'})$ has $\br{\cdot}$-invariant entries that belong to $\field_0 \cap \field_\infty$ and is the identity matrix at $\u = 0$ and $\u = \infty$.
\end{theorem}
\begin{proof}
First note that the $\br{\cdot}$-invariance of the lower and upper canonical basis of $M_\lambda$ shows that the entries of $S(\lambda)$ are $\br{\cdot}$-invariant. As remarked after Definition \ref{d seminormal}, the upper seminormal and lower seminormal bases differ by a diagonal transformation. Thus $S(\lambda) = T(\lambda) D(\lambda) T'(\lambda)^{-1}$ for some diagonal matrix $D(\lambda)$. Given Proposition \ref{p gt to canonical transition}, it suffices to show that $D(\lambda)$ is the identity matrix at $\u = 0$.

Let $A^s$ (resp. $A'^s$) be the matrix that expresses right multiplication by $C_s$ in terms of the upper (resp. lower) seminormal basis of $M_\lambda$. Then by definition of $D(\lambda)$, $D(\lambda) A^s D(\lambda)^{-1} = A'^s$.  Also, it follows from Proposition \ref{p dual bases Mlambda} that $A^s = \transpose{(A'^s)}$. Thus $D(\lambda)$ is determined up to a global scale by the equations $\frac{D(\lambda)_{Q'Q'}}{D(\lambda)_{QQ}} = \frac{A^s_{QQ'}}{A^s_{Q'Q}}$ for all $s \in S$, $Q,Q' \in \text{SYT}(\lambda)$ such that $A^s_{Q'Q} \neq 0$ ($D(\lambda)$ must be determined uniquely by these equations up to a global scale because $M_\lambda$ is irreducible).

Now $A^s = T(\lambda)^{-1}M^sT(\lambda)$, where $M^s$ expresses right multiplication by $C_s$ in terms of the upper canonical basis of $M_\lambda$ (thus $M^s_{Q'Q} = \mu(Q',Q)$ if $s \in R(C_{Q'}) \setminus R(C_Q)$). Now we apply an easy modification of \eqref{en matrix product mu} to the product $-\u A^s = T(\lambda)^{-1} (-\u M^s) T(\lambda)$ to obtain (assuming  $Q' \neq Q$)
\be
\label{e mu computation}
\begin{array}{c}
\mu(-\u A^s_{Q'Q}) =\chi\{ s \in R(C_Q) \}(\mu(T^{-1}_{Q'Q})) + \mu(-\u M^s_{Q'Q}) + \chi\{ s \in R(C_{Q'}) \}(\mu(T_{Q'Q}))
\vspace{7pt} \\
   =
\begin{cases}
0+\mu(-\u M^s_{Q'Q}) + 0 = -\mu(Q',Q) & \text{if } s \in R(C_{Q'}) \setminus R(C_Q) \text{ and } Q' \ldneq Q,\\
-\mu(Q',Q) + 0 + 0 \ \ = -\mu(Q',Q)& \text{if } s \in R(C_{Q}) \setminus R(C_{Q'}), \\
 & (R(C_{Q'}) \setminus R(C_{Q})) \cap  J_{k(Q',Q)-1} \neq \emptyset, \text{ and } Q' \gdneq Q,
\end{cases}
\end{array}
\ee
where $\chi\{P\}(x)$ is equal to $x$ if $P$ is true and 0 otherwise. Here we have used the lower triangularity of  $T(\lambda)$ for the top case and Proposition \ref{p gt to canonical transition} (iv) for the bottom case (note that these cases do not cover all possibilities, and we do not know the answer in general).

To complete the proof, consider the dual Knuth equivalence graph on  SYT$(\lambda)$ as in, for instance, \cite{Sami}.  We say that a dual Knuth transformation $Q' \underset{\text{DKE}}{\longleftrightarrow} Q$ is \emph{initial} if $Q'$ and $Q$ have the entries $i,i+1$ in different positions and $|R(C_{Q'}) \cap \{s_{i-1}, s_i\}| = |R(C_{Q}) \cap \{s_{i-1}, s_i\}| = 1$.  For example, the dual Knuth equivalence
\setlength{\cellsize}{8pt}
${\tiny \tableau{1 &2 &4 \\ 3 &  5&6} \underset{\text{DKE}}{\longleftrightarrow} \tableau{1 &2 &5 \\ 3 &  4&6}}$
is not initial.  It is easy to check that $Q' \gdneq Q$ and $Q' \underset{\text{DKE}}{\longleftrightarrow} Q$ initial implies $s_{k(Q',Q)-2} \in (R(C_{Q'}) \setminus R(C_{Q})) \cap  J_{k(Q',Q)-1}$.  Hence by applying \eqref{e mu computation} to $A^s_{Q'Q}$ and $A^{s}_{QQ'}$ for any pair $Q', Q$ such that  $Q' \gdneq Q$ and $Q' \underset{\text{DKE}}{\longleftrightarrow} Q$ is initial, and with $s \in R(C_{Q}) \setminus R(C_{Q'})$, we conclude $\frac{D(\lambda)_{Q'Q'}}{D(\lambda)_{QQ}} = \frac{A^s_{QQ'}}{A^s_{Q'Q}} \equiv 1 \mod \u\field_0$.
The result then follows from the following combinatorial claim
\refstepcounter{equation}
\begin{enumerate}[label={(\theequation)}]
\item  The graph on SYT$(\lambda)$ consisting of initial  dual Knuth transformations is connected.
\end{enumerate}

The claim is proved by induction on $r = |\lambda|$.  Let $\mu^1,\ldots,\mu^l$ be the shapes obtained from $\lambda$ by removing an outer corner.  Assume that the graphs for the $\mu^i$ are connected.  For any distinct $i, j \in [l]$, it is easy to construct $Q',Q \in \text{SYT}(\lambda)$ such that $\sh(Q'|_{[r-1]}) = \mu^i$, $\sh(Q|_{[r-1]}) = \mu^j$, and  $Q' \underset{\text{DKE}}{\longleftrightarrow} Q$ is a dual Knuth transformation.  Such dual Knuth transformations are always initial, so the claim follows.
\end{proof}

\subsection{Positivity conjectures}
In our computations of many of the matrices discussed above, we have observed positivity properties, which we make precise below.
Computing in Magma, we have verified (a) and (b) for all  $\lambda \vdash r$, $r \leq 8$ and  (c) for $r \leq 6$.
Our original motivation for looking for positivity here is that the positivity of $S(\lambda)$ is related to the conjecture in \cite{GCT4} stating that an element spanning $ \nsbr{\epsilon}_+ \subseteq \field \nsH_r \subseteq \field \H_r \tsr \H_r$ has nonnegative coefficients when expressed in the basis $\{C_v \tsr C_w : v,w \in \S_r\}$ (see the introduction).
\begin{conjecture}
\label{cj non-negativity T T' D}
For a non-zero matrix $M$ with  $\br{\cdot}$-invariant entries in $\field$, let $D(M)$ be the unique up to sign element of $\field$ such that  $D(M)M$ has $\br{\cdot}$-invariant entries in  $\mathbf{A}$ and the greatest common divisor of the entries in  $D(M)M$ is  $1$. The matrices $\liftT, \liftT', T(\lambda), T'(\lambda)^{-1}$, and $S(\lambda)$ from Corollary \ref{c lift at u = 0 in H}, Proposition \ref{p gt to canonical transition}, and Theorem \ref{t transition C' to C} have the following positivity properties.
\begin{list}{\emph{(\alph{ctr})}} {\usecounter{ctr} \setlength{\itemsep}{1pt} \setlength{\topsep}{2pt}}
\item Let $M$ be $T(\lambda), T'(\lambda)^{-1},$ or $S(\lambda)$. After replacing  $D(M)$ with $-D(M)$ if needed, all of the entries of $D(M)M$ have nonnegative coefficients.
\item If $M$ is $T(\lambda)$ or $T'(\lambda)^{-1}$, then $D(M)$ belongs to $\mathbf{A}$ and has all nonnegative or all nonpositive coefficients (this is not a sensible conjecture for  $S(\lambda)$ because it is only well-defined up to a global scale).
\item $\pm D(\liftT) = \pm D(\liftT') = \pm[r]!$.
\end{list}
\end{conjecture}
It follows from Proposition \ref{p dual bases Mlambda} that $T'(\lambda)^{-1} =\transpose{T(\lambda)}$, so the nonnegativity conjectures for $T(\lambda)$ and $T'(\lambda)^{-1}$ are equivalent. This conjecture, or rather its weakening discussed in the remark below, is supported by Proposition \ref{p gt to canonical transition} (iv) since the $\S_r$-graph edge weights $\mu(Q',Q)$ are known to be nonnegative.

%This is for the old part (a), which is now known to be false
%Part (a) is also supported by a result for the two-row case (Theorem \ref{t two row partial lifts}, below). We have not been able to figure out a rule for the signs in $\liftT$ and $\liftT'$ or for the matrices $\tilde{M}^{J_i}$ described in Lemma \ref{l lift transition matrix}.
%%I'm hiding here fact that liftT and liftT' are related by -1^ length difference
%Interestingly, though the  $\tilde{M}^{J_i}$ only appear to have to positivity property (a), the product $T(\lambda) = \tilde{M}^{J_{r-1}} \tilde{M}^{J_{r-2}} \dots \tilde{M}^{J_{1}}$ appears to have the stronger positivity property (b).

\begin{remark}
It is not completely clear how to define nonnegativity in $\field$.  At first, we used the following definition of nonnegativity: $f \in \field$ is nonnegative if $f = g/h$, $g, h \in \mathbf{A}$, $g$ and $h$ have no common factor, and $g$ and $h$ have nonnegative coefficients. To our surprise, we discovered that this is not a good definition because this subset is not a semiring. For example, $[2],[3],$ and $\frac{1}{[6]}$ are all nonnegative by this definition, but $\frac{[2][3]}{[6]} = \frac{\u^2}{1-\u^2+\u^4}$ is not (in fact, this is an entry of $T'((6,2))$).

A strictly weaker definition of nonnegativity that we may adopt instead is: an element $f \in \field$ is \emph{nonnegative} if $f(a)$ is defined and nonnegative for all positive real $a$. With this definition, the set of nonnegative rational functions in  $\u$ is a semiring and Conjecture \ref{cj non-negativity T T' D} would imply that the matrices $T(\lambda), T'(\lambda)^{-1}$, and $S(\lambda)$ (after adjusting $S(\lambda)$ by a suitable global scale) have nonnegative entries.
\end{remark}

\begin{remark}
It is tempting to conjecture from Figure \ref{l lift transition matrix} that every entry of $D(\liftT)\liftT$ has either all nonnegative coefficients or all nonpositive coefficients. This turns out to be true for $r \leq 5$, but fails for $r = 6$---the only entries of  $[6]! \, \liftT$ without this property are equal to
\[ [2]^3[5]([3]-3) = \u^9 +2\u^7-2 \u^3-\u-\ui -2\u^{-3}+2\u^{-7}+\u^{-9}.\]
Despite this failure, the matrices $\liftT, \liftT'$ deserve further investigation as their entries appear combinatorial in nature.
\end{remark}
\section{The two-row case}
\label{s The two-row case}
We now set $n=2$ and use the graphical calculus of \cite{FK} to compute the transition matrix of Lemma \ref{l lift transition matrix} explicitly for $\lambda$ a two-row partition.

\begin{definition}
\label{d pairing internal external}
The \emph{diagram} of a word $\mathbf{k} \in [n]^r$ is the picture obtained from $\mathbf{k}$ by pairing 2s and 1s as left and right parentheses and then drawing an arc between matching pairs as shown below.
%in Figure \ref{f standard diagram calculus}.
The word  $\mathbf{k}$ is \emph{Yamanouchi} if its diagram has no unpaired 2s.
\end{definition}
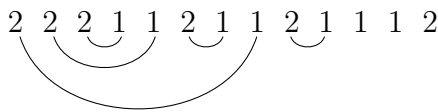
\begin{figure}[H]
\begin{tikzpicture}[xscale=.9]
\tikzstyle{column} = [inner sep = -4pt]
\tikzstyle{edge} = [draw,-,black]
\tikzstyle{LabelStyleH} = [text=black, fill =white, inner sep = -.8pt]
\begin{scope}[yshift=0cm]
\foreach \y/\z in {-.2/-14*.25} {
    \draw[edge, bend right=70] (\z+1*.5,\y) to (\z+8*.5,\y);
    \draw[edge, bend right=70] (\z+2*.5,\y) to (\z+5*.5,\y);
    \draw[edge, bend right=70] (\z+3*.5,\y) to (\z+4*.5,\y);
    \draw[edge, bend right=70] (\z+6*.5,\y) to (\z+7*.5,\y);
    \draw[edge, bend right=70] (\z+9*.5,\y) to (\z+10*.5,\y);
    \node[column] (theNode) at (0,0) {
    $\myvcenter{\ensuremath{\pad{2} \pad{2} \pad{2} \pad{1} \pad{1} \pad{2} \pad{1} \pad{1} \pad{2} \pad{1} \pad{1} \pad{1} \pad{2} }}$
    };
}
\end{scope}
\vspace{-.4in}
\end{tikzpicture}
%\caption{The diagram of 2221121121112.}
%\label{f standard diagram calculus}
\end{figure}
As shown in \cite{FK}, diagrams provide a  simple and beautiful way to visualize the action of $\Uq$ and  $\H_r$ on the upper canonical basis.
%as well as the expansion of the upper canonical basis in terms of the monomial basis $\{\bv_\mathbf{k}: \mathbf{k}\in [n]^r\}$ can be described simply and beautifully in terms of diagrams.
The first part of the next theorem is established in \cite[\textsection 2.3]{FK}, and the second part is obtained from the first by dualizing with respect to the inner product on $\bT$.

\begin{theorem}\label{t FK F action on c basis}
\
\begin{list}{\emph{(\alph{ctr})}} {\usecounter{ctr} \setlength{\itemsep}{1pt} \setlength{\topsep}{2pt}}
\item The action of $F_1$ on the upper canonical basis of $\bT$ is given by
\[F_1 (c_{\mathbf{k}}) = \sum_{j=1}^t [j]c_{\F_{(j)}(\mathbf{k})}, \]
where $t$ is the number of unpaired $1$s in $\mathbf{k}$ and $\F_{(j)}(\mathbf{k})$ is the word obtained by replacing the $j$-th unpaired  $1$ in  $\mathbf{k}$ with a $2$ (the first unpaired  $1$ means the leftmost unpaired 1 and the $t$-th means the rightmost).
\item the action of $E_1$ on the lower canonical basis of $\bT$ is given by
\[ E_1 (\cp_{\mathbf{k}^\dagger}) = \sum_{\mathbf{k}'} [\alpha(\mathbf{k'},\mathbf{k})]\cp_{\mathbf{k'}^\dagger}, \]
where $\alpha(\mathbf{k}',\mathbf{k})$ is the positive integer $j$ such that $\F_{(j)}(\mathbf{k}') = \mathbf{k}$ and 0 if there is no such positive integer.
\end{list}
\end{theorem}

We will also make use of the action of the Kashiwara operator $\crystalu{F_1}$ on the upper crystal basis (we abuse notation  by letting the operator act on words rather than the crystal basis elements $b_\mathbf{k} \in \B$):
% and $b'_\mathbf{k}  \in  \B' \subseteq \L_0/ \u \L_0$):
%\refstepcounter{equation}
\begin{enumerate}%[label={(\theequation)}]
\item[]  $\crystalu{F_1}(\mathbf{k})$  is the word obtained by replacing the rightmost unpaired 1 in  $\mathbf{k}$ with a 2 and is undefined if there are no unpaired 1s.
%\item  $\crystall{F_1}(\mathbf{k}) = \crystalu{F_1}(\mathbf{k}^\dagger)^\dagger$.
\end{enumerate}

We need some notation for the next theorem. Let $\mathbf{k}|_j$ denote the subword $k_1k_2 \cdots k_j$ of the word $\mathbf{k} = k_1k_2\cdots k_r$. Let $f$ denote the function on $V^{\tsr r-1}$ given by
\[ f(\sum_{\mathbf{j} \in [n]^{r -1}} a_\mathbf{j} \cp_{\mathbf{j}^\dagger}) =  \sum_{\mathbf{j}} a_\mathbf{j} \cp_{\crystalu{F_1}(\mathbf{j}^\dagger)}, \]
where the  $a_\mathbf{j}$ belong to $\field$ and the sum on the right is over those $\mathbf{j}$ such that  $\crystalu{F_1}(\mathbf{j}^\dagger)$ is defined.

Let $\lambda$ be a partition of $r$ with two rows and identify the lower canonical basis of $M_\lambda$ with the $\H_r$-cell $\Gamma'_{Z_\lambda}$ of $\bT$ (the vertices of this cell are those $\cp_{\mathbf{k}^\dagger}$ such that $\mathbf{k}$ is Yamanouchi and has content $\lambda$) via the right-hand isomorphism of \eqref{e sb involution} (with  $T = Z_\lambda$).
Set $\lambda^1 = (\lambda_1 -1, \lambda_2)$ and $\lambda^2 = (\lambda_1, \lambda_2-1)$ and $l = \lambda_1 - \lambda_2$.
Let $\crystalu{F_1}(Z_{\lambda^2})$ denote the tableau obtained from $Z_{\lambda^2}$ by changing the last entry in the first row to a 2.

We will compute the transition matrix of Lemma \ref{l lift transition matrix} for $\lambda$ as above.  We have found it more convenient to compute the matrix for  $J^\dagger = \{s_2,\ldots,s_{r-1}\}$ rather than $J = \{ s_1, \dots, s_{r-2} \}$ (the matrix for $J$ can then be obtained from that for $J^\dagger$ by conjugating by the permutation matrix corresponding to $\C_Q \mapsto \C_{Q^\dagger}$).
Consider the weight space $\bT^\lambda$, which is isomorphic to $\epsilon_+ \tsr_{\field \H_{J_\lambda}} \field \H_r$.  Since the intersection of two cellular subquotients is a cellular subquotient, Proposition \ref{p restrict Wgraph} with parabolic subgroup $(\S_{r})_{J^\dagger}$ and Theorem \ref{t Schur-Weyl duality lower} imply that
\be R: Res_{\field \H_{J^\dagger}} \field \{\cp_{\mathbf{k}^\dagger} \in \bT^\lambda : k_r = 1\} \xrightarrow{\cong} (V^{\tsr r-1})^{\lambda^1}
, \ \cp_{\mathbf{k}^\dagger} \mapsto \cp_{(\mathbf{k}|_{r-1})^\dagger} \ee
is an isomorphism of $\field \H_{J^\dagger}$-modules with basis.  Quotienting  by  $\H_r$-cells below  $\Gamma'_{Z_\lambda}$, this yields an isomorphism of modules with basis
\be \label{e R lambda definition}
Res_{\field \H_{J^\dagger}} \field\Gamma'_{Z_\lambda} \xrightarrow{\cong} \field \bigl( \Gamma'_{Z_{\lambda^1}} \sqcup \Gamma'_{\crystalu{F_1}(Z_{\lambda^2})} \bigr). \ee
%this amounts to intersecting the cellular subquotient 1( ...) of Res \Gamma'_\S_r with the induced \bT^\lambda
%I checked this
\begin{theorem}
\label{t two row partial lifts}
Maintain the notation above.
For each $\cp_{\mathbf{k}^\dagger} \in \Gamma'_{Z_\lambda}$, define the element
\be \label{e liftcp definition}
(\liftcp_{\mathbf{k}^\dagger})^{J^\dagger} := \begin{cases} \displaystyle
\cp_{\mathbf{k}^\dagger} - \frac{1}{[l+1]} R^{-1}( f (E_1(\cp_{(\mathbf{k}|_{r-1})^\dagger}))) & \text{if } \sh((\mathbf{k}|_{r-1})^\dagger) = \lambda^1, \\
\cp_{\mathbf{k}^\dagger} & \text{if } \sh((\mathbf{k}|_{r-1})^\dagger) = \lambda^2.
\end{cases}
\ee
Applying $s_\lambda^{\bT^\lambda}$ to both sides of this definition ($s_\lambda^N$ is the surjection onto the  $M_\lambda$-isotypic component of $N$)
then yields $(\liftCp_{Q(\mathbf{k})^\dagger})^{J^\dagger}$ expanded in terms of the lower canonical basis of  $M_\lambda$.
%this last fact uses the right-hand isomorphism of \eqref{e sb involution}
\end{theorem}
It is helpful to follow the proof with an example: take $\lambda = (5,2)$ and  $\mathbf{k} = 2112111$. Then
\[ E_1(\cp_{211211^\dagger}) = \cp_{111211^\dagger} + [2]\cp_{211111^\dagger},  \]
\[ (\liftcp_{2112111^\dagger})^{J^\dagger} = \cp_{2112111^\dagger} - \frac{1}{[4]}(\cp_{1112121^\dagger} + [2]\cp_{2111121^\dagger}), \]
\setlength{\cellsize}{8pt}
\[\Big(\liftCp_{{\tiny\tableau{1 & 2 & 3 &5 &6 \\ 4 & 7}}}\Big)^{J^\dagger} = \C_{{\tiny\tableau{1 & 2 & 3 &5 &6 \\ 4 & 7}}} - \frac{1}{[4]}\Big(\C_{{\tiny\tableau{1 & 3 & 5 &6 &7 \\ 2 & 4}}} + [2]\C_{{\tiny\tableau{1 & 3 & 4 &5 &6 \\ 2 & 7}}}\Big), \]
\[\Big(\liftCp_{{\tiny\tableau{1 & 3 & 4 &6 &7 \\ 2 & 5}}}\Big)^{J} = \C_{{\tiny\tableau{1 & 3& 4 &6 &7 \\ 2 & 5}}} - \frac{1}{[4]}\Big(\C_{{\tiny\tableau{1 & 2 & 3 &4 &6 \\ 5 & 7}}} + [2]\C_{{\tiny\tableau{1 & 3 & 4 &5 &6 \\ 2 & 7}}}\Big). \]
\begin{proof}
Assume throughout that  $\mathbf{k}$ corresponds to the top case of \eqref{e liftcp definition}, the arguments needed for the bottom case being easy.  The key fact to check is that $E_1(R((\liftcp_{\mathbf{k}^\dagger})^{J^\dagger}))$ is zero mod $V^{\tsr r - 1}[\gdneq \lambda^2]$.   To see that this would prove the theorem, let $\eta$ be the element of the weight space $(V^{\tsr r-1})^{\lambda^1}$ such that $E_1(\eta) = 0$ and  $R((\liftcp_{\mathbf{k}^\dagger})^{J^\dagger}) - \eta \in V^{\tsr r - 1}[\gdneq \lambda^2]$.   Then  $\eta$ is a highest weight vector of weight $\lambda^1$, so by quantum Schur-Weyl duality, $\eta$ belongs to the $M_{\lambda^1}$-isotypic component of  $V^{\tsr r -1}$.  Thus $R((\liftcp_{\mathbf{k}^\dagger})^{J^\dagger})$   and  $\eta$ only differ by lower canonical basis elements outside of  $ \Gamma'_{Z_{\lambda^1}} \sqcup \Gamma'_{\crystalu{F_1}(Z_{\lambda^2})}$; so by \eqref{e R lambda definition}, $(\liftcp_{\mathbf{k}^\dagger})^{J^\dagger}$, regarded as an element of the cellular subquotient  $\field\Gamma'_{Z_\lambda}$, belongs to the $M_{\lambda^1}$-isotypic component of $\Res_{\field \H_{J^\dagger}} \field\Gamma'_{Z_\lambda}$.

Now, checking the key fact amounts to showing that if
\[ (E_1 - \frac{E_1}{[l+1]} f E_1)(\cp_{(\mathbf{k}|_{r-1})^\dagger}) = (1 - \frac{E_1}{[l+1]} f) \sum_{\mathbf{j}'} [\alpha(\mathbf{j'},\mathbf{k}|_{r-1})]\cp_{\mathbf{j'}^\dagger} \]
is written as $\sum_{\mathbf{j} \in [n]^{r-1}} a_\mathbf{j} \cp_{\mathbf{j}^\dagger}$, then $a_\mathbf{j}= 0$ for $\mathbf{j}$ such that
$\mathbf{j}$ is Yamanouchi. Here we are using the fact that  $(V^{\tsr r - 1})^{\lambda^2}[\gdneq \lambda^2]$ is spanned by $\cp_{\mathbf{j}^\dagger}$ such that $\mathbf{j}$ has content $\lambda^2$ and is not Yamanouchi.
Now let $\mathbf{j}$ be of content $\lambda^2$ and Yamanouchi; then one checks that $\cp_{\mathbf{j}^\dagger}$ occurs as a term of $E_1 f \cp_{\mathbf{j'}^\dagger}$ expanded in the lower canonical basis if and only if  $\mathbf{j} = \mathbf{j'}$, and if it does occur, then its coefficient is  $[l+1]$ since  $l+1$ is the number of unpaired 1s in  $\mathbf{j}$.  It follows that $a_\mathbf{j} = 0$, as desired.
\end{proof}

\begin{remark}\label{r Lascoux's paper}
The recent paper \cite{GLS} studies the matrix $T'(\lambda)$ for $\lambda$ a two-row partition.  The lower canonical basis of $M_\lambda$ is realized in a polynomial representation of $\H_r$ and the lower seminormal basis of $M_\lambda$ is given by specialized non-symmetric Macdonald polynomials. Let $\lambda = (r/2,r/2)$ and $Q$ be the SYT of shape $\lambda$ such that the first row of $Q$ has odd entries and the second has even entries.  The authors show that the coefficients of $\gtp_{Q}$ expressed in the lower canonical basis of $M_\lambda$ (i.e. the last column of  $T'(\lambda)$) are all powers of $-\frac{1}{[2]}$ and they give a combinatorial formula for the exponents.
% Using \theta , we can show that $T'(\lambda)_{\transposeQ' \transposeQ} = T'(\lambda')^{-1}_{Q' Q} (-1)^{\ell(Q)+\ell(Q')}$, so this supports the positivity conjecture
%with  {\tiny\tableau{1 & 3 & \dots & r-3 & r-1 \\ 2 & 4 & \dots & r-2 & r }}}$
\end{remark}

\section*{Acknowledgments}
I am grateful to John Stembridge, Ketan Mulmuley, and Thomas Lam for helpful conversations and to Michael Bennett for  help typing and typesetting figures.

\bibliographystyle{plain}
%\bibliography{mycitations}

\begin{thebibliography}{10}

\bibitem{Sami}
Sami~H. {Assaf}.
\newblock {Dual Equivalence Graphs I: A combinatorial proof of LLT and
  Macdonald positivity}.
\newblock {\em ArXiv e-prints}, May 2010.
\newblock {\tt arXiv:1005.3759}.

\bibitem{BV}
Dan Barbasch and David Vogan.
\newblock Primitive ideals and orbital integrals in complex exceptional groups.
\newblock {\em J. Algebra}, 80(2):350--382, 1983.

\bibitem{BMSGCT4}
J.~{Blasiak}, K.~D. {Mulmuley}, and M.~{Sohoni}.
\newblock {Geometric Complexity Theory IV: nonstandard quantum group for the
  Kronecker problem}.
\newblock {\em {A}r{X}iv e-prints}, June 2013.
\newblock {\tt arXiv:cs/0703110v4}.

\bibitem{Bnsbraid}
Jonah Blasiak.
\newblock {Nonstandard braid relations and Chebyshev polynomials}.
\newblock {\em ArXiv e-prints}, October 2010.
\newblock {\tt arXiv:1010.0421}.

\bibitem{B0}
Jonah Blasiak.
\newblock {$W$}-graph versions of tensoring with the {$\S_n$} defining
  representation.
\newblock {\em J. Algebraic Combin.}, 34(4):545--585, 2011.

\bibitem{B4}
Jonah {Blasiak}.
\newblock {Representation theory of the nonstandard Hecke algebra}.
\newblock {\em ArXiv e-prints}, February 2012.
\newblock {\tt arXiv:1201.2209v2}.

\bibitem{Brundan}
Jonathan Brundan.
\newblock Dual canonical bases and {K}azhdan-{L}usztig polynomials.
\newblock {\em J. Algebra}, 306(1):17--46, 2006.

\bibitem{GLS}
J.~{de Gier}, A.~{Lascoux}, and M.~{Sorrell}.
\newblock {Deformed Kazhdan-Lusztig elements and Macdonald polynomials}.
\newblock {\em ArXiv e-prints}, July 2010.
\newblock {\tt arXiv:1007.0861}.

\bibitem{Du}
Jie Du.
\newblock {${\rm IC}$} bases and quantum linear groups.
\newblock In {\em Algebraic groups and their generalizations: quantum and
  infinite-dimensional methods ({U}niversity {P}ark, {PA}, 1991)}, volume~56 of
  {\em Proc. Sympos. Pure Math.}, pages 135--148. Amer. Math. Soc., Providence,
  RI, 1994.

\bibitem{F}
Sergey Fomin.
\newblock {\em \emph{Knuth equivalence, jeu de taquin, and the
  Littlewood-Richardson rule}, Appendix {I} in Enumerative Combinatorics, vol.
  2}, volume~62 of {\em Cambridge Studies in Advanced Mathematics}.
\newblock Cambridge University Press, Cambridge, 1999.

\bibitem{FKK}
I.~B. Frenkel, M.~G. Khovanov, and A.~A. Kirillov, Jr.
\newblock Kazhdan-{L}usztig polynomials and canonical basis.
\newblock {\em Transform. Groups}, 3(4):321--336, 1998.

\bibitem{FK}
Igor~B. Frenkel and Mikhail~G. Khovanov.
\newblock Canonical bases in tensor products and graphical calculus for
  {$U_q(\sl_2)$}.
\newblock {\em Duke Math. J.}, 87(3):409--480, 1997.

\bibitem{GM}
A.~M. Garsia and T.~J. McLarnan.
\newblock Relations between {Y}oung's natural and the {K}azhdan-{L}usztig
  representations of {$S_n$}.
\newblock {\em Adv. in Math.}, 69(1):32--92, 1988.

\bibitem{GL}
Ian Grojnowski and George Lusztig.
\newblock On bases of irreducible representations of quantum {${\rm GL}\sb n$}.
\newblock In {\em Kazhdan-{L}usztig theory and related topics ({C}hicago, {IL},
  1989)}, volume 139 of {\em Contemp. Math.}, pages 167--174. Amer. Math. Soc.,
  Providence, RI, 1992.

\bibitem{Himmanant}
Mark Haiman.
\newblock Hecke algebra characters and immanant conjectures.
\newblock {\em J. Amer. Math. Soc.}, 6(3):569--595, 1993.

\bibitem{HK}
Jin Hong and Seok-Jin Kang.
\newblock {\em Introduction to quantum groups and crystal bases}, volume~42 of
  {\em Graduate Studies in Mathematics}.
\newblock American Mathematical Society, Providence, RI, 2002.

\bibitem{HY1}
Robert~B. Howlett and Yunchuan Yin.
\newblock Inducing {$W$}-graphs.
\newblock {\em Math. Z.}, 244(2):415--431, 2003.

\bibitem{HY2}
Robert~B. Howlett and Yunchuan Yin.
\newblock Inducing {$W$}-graphs. {II}.
\newblock {\em Manuscripta Math.}, 115(4):495--511, 2004.

\bibitem{Jimbo}
Michio Jimbo.
\newblock A {$q$}-analogue of {$U(\gl(N+1))$}, {H}ecke algebra, and the
  {Y}ang-{B}axter equation.
\newblock {\em Lett. Math. Phys.}, 11(3):247--252, 1986.

\bibitem{Joseph}
Anthony Joseph.
\newblock Towards the {J}antzen conjecture. {III}.
\newblock {\em Compositio Math.}, 42(1):23--30, 1980/81.

\bibitem{Kas1}
Masaki Kashiwara.
\newblock On crystal bases of the {$Q$}-analogue of universal enveloping
  algebras.
\newblock {\em Duke Math. J.}, 63(2):465--516, 1991.

\bibitem{Kas2}
Masaki Kashiwara.
\newblock Global crystal bases of quantum groups.
\newblock {\em Duke Math. J.}, 69(2):455--485, 1993.

\bibitem{KL}
David Kazhdan and George Lusztig.
\newblock Representations of {C}oxeter groups and {H}ecke algebras.
\newblock {\em Invent. Math.}, 53(2):165--184, 1979.

\bibitem{L2}
George Lusztig.
\newblock Cells in affine {W}eyl groups.
\newblock In {\em Algebraic groups and related topics ({K}yoto/{N}agoya,
  1983)}, volume~6 of {\em Adv. Stud. Pure Math.}, pages 255--287.
  North-Holland, Amsterdam, 1985.

\bibitem{LBook}
George Lusztig.
\newblock {\em Introduction to quantum groups}, volume 110 of {\em Progress in
  Mathematics}.
\newblock Birkh\"auser Boston Inc., Boston, MA, 1993.

\bibitem{GCT4}
Ketan Mulmuley and Milind~A. Sohoni.
\newblock Geometric complexity theory {IV}: {Q}uantum group for the {K}ronecker
  problem.
\newblock {\em CoRR}, abs/cs/0703110, 2007.

\bibitem{Ram}
Arun Ram.
\newblock A {F}robenius formula for the characters of the {H}ecke algebras.
\newblock {\em Invent. Math.}, 106(3):461--488, 1991.

\bibitem{RamSeminormal}
Arun Ram.
\newblock Seminormal representations of {W}eyl groups and {I}wahori-{H}ecke
  algebras.
\newblock {\em Proc. London Math. Soc. (3)}, 75(1):99--133, 1997.

\bibitem{R}
Yuval Roichman.
\newblock Induction and restriction of {K}azhdan-{L}usztig cells.
\newblock {\em Adv. Math.}, 134(2):384--398, 1998.

\bibitem{Wenzl}
Hans Wenzl.
\newblock Hecke algebras of type {$A_n$} and subfactors.
\newblock {\em Invent. Math.}, 92(2):349--383, 1988.

\end{thebibliography}
\def\cprime{$'$}

\end{document}